\documentclass[11pt,twoside]{article}
\usepackage{amsmath}
\usepackage{amsthm}
\usepackage{amsfonts}
\usepackage{amssymb}
\usepackage{mathtools}
\usepackage[latin1]{inputenc}
\usepackage{color} 
\numberwithin{equation}{section}

\newtheoremstyle{thmstyle}
  {3pt}{3pt}{\itshape}{}{\bfseries}{}{.5em}
  {\thmname{#1}\thmnumber{ #2}\thmnote{ \textmd{(#3)}}}
\theoremstyle{thmstyle}

\newtheorem{teo}{Theorem}[section]
\newtheorem{lemma}[teo]{Lemma}
\newtheorem{coro}[teo]{Corollary}
\newtheorem{prop}[teo]{Proposition}
\newtheorem{De}[teo]{Definition}
\newtheorem{Rem}[teo]{Remark}
\newtheorem{Ex}[teo]{Example}
\newenvironment{de}{\begin{De}\rm}{\end{De}}

\newenvironment{re}{\begin{Rem}\rm}{\end{Rem}}

\title{ Manifolds of mappings associated with real-valued functions spaces}
\author{ author }

\newcommand{\R}{\mathbb{R}}
\newcommand{\Rn}{\mathbb{R}^n}
\newcommand{\Rm}{\mathbb{R}^m}

\newcommand{\F}{\mathcal{F}}
\newcommand{\N}{\mathbb{N}}
\newcommand{\gf}{\Gamma_\mathcal{F}(\eta)}

\begin{document}
\begin{center}
\textbf{\Large Manifolds of mappings associated with real-valued function spaces and natural mappings between them}\\[2mm]
\textbf{Matthieu F. Pinaud}
\end{center}
\begin{abstract}
\noindent
Let $M$ be a compact smooth manifold with corners and $N$ be a finite dimensional
smooth manifold without boundary which admits local addition. 
We define a smooth manifold structure to general sets of continuous 
mapings $\mathcal{F}(M,N)$ whenever functions spaces $\mathcal{F}(U,\R)$ 
on open subsets $U\subseteq [0,\infty)^n$ are given, subject to simple axioms.
Construction and properties of spaces of sections and 
smoothness of natural mappings between spaces $\mathcal{F}(M,N)$ are discussed, 
like superposition operators $\mathcal{F}(M,f):\mathcal{F}(M,N_1)\to \mathcal{F}(M,N_2)$, 
$\eta \mapsto f\circ \eta$ for smooth maps $f:N_1\to N_2$.
\end{abstract}
\textbf{MSC 2020 subject classification}: 58D15 (primary), 46T10, 46T20, 46T05 (secondary)\\
\textbf{Keywords}: manifold of mappings; infinite-dimensional manifold; pushforward; superposition operator; Nemytskij operator 

\section{Introduction}
Following the work of H. Gl\"ockner and L. Tarrega \cite{FGL1}, in this article we describe a general construction principle for smooth manifold structures on sets of mappings between manifolds when real-valued functions spaces are given, satisfying suitable axioms. The modeled space of these manifolds structures, which coincide with the space of sections, are studied at the beginning. Then we study the construction and properties of natural mappings between these manifolds of mappings. 

\noindent
For fixed $m,n \in \N$, we consider a $m$-dimensional compact smooth manifold with corners $M$ and $N$ be a $n$-dimensional smooth manifold without boundary. We consider a basis of the topology $\mathcal{U}$ of the set $[0,\infty)^m$ satisfying suitable properties (see Definition \ref{pre goodcollection}). Suppose that for each open set $U\in \mathcal{U}$, an integral complete locally convex space $\F(U,\R)$ of bounded, continuous real-valued functions are given. Then for each finite-dimensional real vector space $E$, the set of maps $\F(U,E)$ can be defined in a natural way. If certain axioms are satisfied (see Definition \ref{pre familysuitable}), we say that the family $(\F(U,E))_{U\in \mathcal{U}}$ is suitable for global analysis. As direct consequence of the case where $M$ is a smooth manifold without boundary (see \cite{FGL1}), it can be shown that one can define a locally convex space $\F(M,E)$. Moreover, we can also define the set of mappings $\F(M,N)$ of $N$-valued functions on the manifold with corners $M$. 

\noindent
Let us fix notation.
\begin{de}
Let $N$ be a smooth manifold and $\pi_{N}:TN\to N$ its tangent bundle. A local addition is a smooth map ${\Sigma:\Omega\to N}$ defined on a open neightborhood $\Omega \subseteq TN$ of the zero-section ${0_N := \{ 0_p \in T_p N : p\in N\}}$ such that $\Sigma(0_p)=p$ for all $p\in N$ and the image $\Omega^\prime:= \big( \pi_{N},\Sigma\big) (\Omega)$ is open in $N\times N$ and the map $\theta_N:=(\pi_{N},\Sigma):\Omega\to \Omega^\prime$ is a $C^\infty$-diffeomorphism.
\end{de}
\noindent
For each function $\gamma:M\to N$ in $\F(M,N)$, we define the real vector space of sections with the pointwise operations
\[ \gf:=\{ \sigma\in \F(M,TN):\pi_{N}\circ \sigma = \gamma\}\]
and we endow it with a  natural topology making it a integral complete locally convex topological
vector space. We define the set
\[ \mathcal{V}_\gamma := \{ \sigma \in \Gamma_\F (\gamma) : \sigma(M)\subseteq \Omega \}. \]
which is open in $\Gamma_\F (\gamma)$. Setting the set
\[ \mathcal{U}_\gamma := \{ \xi \in \Gamma_\F (\gamma) : (\gamma,\xi)(M) \subseteq \Omega' \}.\]
the map
\[   \Psi_\gamma:= \F(M,\Sigma) : \mathcal{V}_\gamma  \to \mathcal{U}_\gamma,\quad \sigma  \mapsto \Sigma\circ \sigma  \]
is a bijection. We show that (see Theorem \ref{main1}):
\begin{teo}
Let $\mathcal{U}$ be a good collection of open subsets. If $(\F(U,\R))_{U\in \mathcal{U}}$ is a family of locally convex space suitable for global analysis, then for each compact manifold $M$ with corners and smooth manifold~$N$ without boundary which admits local addition, the set $\F(M,N)$ 
admits a smooth manifold structure
such that the sets ${\mathcal U}_\gamma$ are open in $\F(M,N)$
for all $\gamma\in \F(M,N)$
and $\Psi_\gamma\colon {\mathcal V}_\gamma\to {\mathcal U}_\gamma$
is a $C^\infty$-diffeomorphism.
\end{teo}
\noindent
Using the smooth manifold structures just described, we find:
\begin{prop}
Let $M$ be a $m$-dimensional compact smooth manifold with corners, $N_1$ and $N_2$ be $n$-dimensional smooth manifold which admits local addition $(\Omega_1,\Sigma_1)$ and $(\Omega_2,\Sigma_2)$ respectively. If $f:N_1\to N_2$ is a $smooth$ map, then the map
\[ \F(M,f):\F(M,N_1)\to \F(M,N_2),\quad \gamma \mapsto f\circ \gamma ,\]
is $smooth$.
\end{prop}

\noindent
In particular, for $p\in M$, the point evaluation map $\varepsilon_p:\F(M,N)\to N$ is smooth (see Proposition \ref{pointeval}). For each $v\in T\F(M,N)$, we define the map
\[ \Theta_N(v) : M \to TN,\quad \Theta_N(v)(p) := T\varepsilon_p (v).\]
Then with respect to the tangent bundle of $\F(M,N)$ we have:
\begin{prop}
Let $M$ be a $m$-dimensional compact smooth manifold with corners, $N$ be a $n$-dimensional smooth manifold which admits a local addition and $\pi_{N}:TN\to N$ its tangent bundle. Then the map
\[ \F(M,\pi_{N}):\F(M,TN)\to \F(M,N),\quad \tau \mapsto \pi_{N}\circ \tau\]
is a smooth vector bundle with fiber $\Gamma_{\F}(\gamma)$ over $\gamma\in \F(M,N)$. Moreover, the map
\[ \Theta_N:T\F(M,N)\to F(M,TN), \quad v \mapsto \Theta_N(v) \]
is an isomorphism of vector bundles.
\end{prop}
\noindent
Let $M$ and $N$ be finite-dimensional smooth manifolds without boundary. For $0<\lambda\leq 1$, we define the set $BC^{0,\lambda}(M,N)$ of all continuous functions $\gamma:M\to N$ such that for charts $\varphi:U\to \varphi(U)$ and $\phi:V\to \phi(V)$ around $p\in M$ and $\gamma(p)\in N$ respectively, such that $\gamma(U)\subseteq V$ and the composition $\phi \circ \gamma\circ \varphi^{-1}:\varphi(U)\to \R^n$ is $\lambda$-H\"older continuous. It is known that $BC^{0,\lambda}(M,N)$ has a smooth manifold structure (see e.g. \cite{Kri}). We will show this fact using the provided construction.

\section{Preliminaries}
\begin{de}\label{pre goodcollection}
Let $m\in \N$ fixed, a set $\mathcal{U}$ of open subsets of product set $[0,\infty)^m$ will be called a \textit{good collection of open subsets} if the following condition are satisfied:
\begin{itemize}
    \item[a)] $\mathcal{U}$ is a basis for the topology of $[0,\infty)^m$.
    \item[b)] If $U\in \mathcal{U}$ and $K\subseteq U$ is a compact non-empty subset, then there exists $V\in \mathcal{U}$ with compact closure $\overline{V}$ in $[0,\infty)^m$ such that $K\subseteq V$ and $\overline{V}\subseteq U$.
    \item[c)] If $U\subseteq [0,\infty)^m$ is an open set and $W\in \mathcal{U}$ is a relatively compact subset of $U$, then there exists $V\in \mathcal{U}$ such that $V$ is a relatively compact subset of $U$ and $\overline{W}\subseteq V$.
    \item[d)] If $\phi:U\to V$ is a $C^\infty$-diffeomorphism between open subsets $U$ and $V$ of $[0,\infty)^m$ and $W \in \mathcal{U}$ is a relatively compact subset of $U$, then $\phi(W)\in \mathcal{U}$.
\end{itemize}
\end{de}

\begin{re}
If we consider $\mathcal{U}=\{ U\cap [0,\infty)^m : U \text{ is open in } \Rm \}$ then $\mathcal{U}$ defines a good collection of open subsets. This is also true for the case of open and bounded subsets of $\Rm$.
\end{re}

\noindent
Let $U$ be a open subset of $[0,\infty)^m$, we write $BC(U,\R)$ for the vector space of all bounded continuous functions $f:U\to \R$ endowed with the supremum norm $\lVert \cdot \lVert_\infty$.

\begin{de}\label{pre mancorners}
Let $M$ be a paracompact Hausdorff topological space. A chart $\phi:U\to V$ is a homeomorphism from an open subset $U\subseteq M$ onto an open subset $V \subseteq [0,\infty)^m$. We say that two charts $\phi_1:U_1\to V_1$ and $\phi_2:U_2\to V_2$ are compatibles if $\phi_1(U_1)\cap \phi_2(U_2)=\emptyset$ or the transition map $\phi_2\circ\phi_1^{-1}: \phi_1\left( U_1\cap U_2\right) \to \phi_2\left( U_1\cap U_2\right)$ is smooth.\\ 
We say that $M$ is an $m$-dimensional smooth manifold with corners if $M$ is equipped with a maximal family of charts $\{ \phi_i:U_i\to V_i\}_{i\in I}$ such that each pair of chart, are compatible and $M=\cup_{i\in I} U_i$.\\
We say that $N$ is a smooth manifold if it is a smooth manifold without boundary.
\end{de}

\noindent
For our context, one important property of smooth manifolds with corners is the existence of cut-off functions. 
\begin{lemma}\label{bumpfunction}
Let $M$ be a $m$-dimensional smooth manifold with corners, $K$ be a closed subset of $M$ and $U$ be a open subset of $M$ containing $K$. Then there exists a smooth function $\xi:M\to [0,1]$ such that $\xi|_K = 1$ and $\text{supp}(\xi)\subseteq U$.
\end{lemma}

\begin{de}\label{pre familysuitable}
Let $\mathcal{U}$ be a good collection of open subsets of $[0,\infty)^m$. For $U\in \mathcal{U}$, the vector subspace $\F(U,\R)$ of $BC(U,\R)$ will denote a integral complete locally convex space such that the inclusion map $\F(U,\R)\to BC(U,\R)$ is continuous. \\
Let $\{b_1,...,b_n\}$ be a basis for a finite dimensional real vector space $E$, we define the space
\[ \F(U,E):=\sum_{i=1}^n \F(U,\R)b_i \]
and we endow it with the the locally convex topology making the map
\begin{equation}\label{2 eq iso}
\F(U,\R)^n\to \F(U,E),\quad (f_1,...,f_n)\mapsto \sum_{i=1}^n f_i b_i 
\end{equation}
an isomorphism of topological vector spaces. \\
\noindent
We say that $\left( \F(U,\R)\right)_{U\in \mathcal{U}} $ is a family of locally convex spaces suitable for global analysis if the following axioms are satisfied for all finite-dimensional real vector spaces $E$ and $F$:
\begin{itemize}
\item[(PF)] \textbf{Pushforward Axiom} For all $U, V\in \mathcal{U}$ such that $V$ is relatively compact in $U$ and each smooth map $f:U\times E \to F$, we have $f_*(\gamma):=f\circ (\text{id}_V,\gamma|_V)\in \F(V,F)$ for all $\gamma\in \F(U,E)$ and the map
\[ f_*:\F(U,E)\to \F(V,F),\quad \gamma\mapsto f\circ (\text{id}_V,\gamma|_V)\]
is continuous.
\item[(PB)]\textbf{Pullback Axiom :} Let $U$ be an open subset of $[0,\infty)^m$ and $V,W\in \mathcal{U}$ such that $W$ has compact closure contained in $U$. Let $\Theta:U\to V$ be a smooth diffeomorphism. Then $\gamma\circ \Theta|_W\in \F(W,E)$ for all $\gamma\in \F(V,E)$ and 
\[ \F(\Theta|_W,E):\F(V,E)\to \F(W,E),\quad \gamma\mapsto \gamma\circ \Theta|_W\]
is continuous.
\item[(GL)]\textbf{Globalization Axiom :} If $U, V\in \mathcal{U}$ with $V\subseteq U$ and $\gamma\in \F(V,E)$ has compact support, then the map $\tilde{\gamma}:U\to E$ defined by
\[ \tilde{\gamma}(x)=\left\{ 
\begin{array}{rl}
\gamma(x), & x\in V \\
0, & x\in U\setminus \text{supp}(\gamma) 
\end{array}
\right.\]
is in $\F(U,E)$ and for each compact subset $K$ of $V$ the map
\[ e_{U,V,K}^E : \F_K (V,E)\to \F(U,E),\quad \gamma \mapsto \tilde{\gamma} \]
is continuous, where $\F_K (V,E):=\{ \gamma\in \F(V,E) : \text{supp}(\gamma)\subseteq K\}$ is endowed with the topology induced by $\F(V,E)$.
\item[(MU)]\textbf{Multiplication Axiom :} If $U\in\mathcal{U}$ and $h\in C_c^{\infty}(U,\R)$, then $h\gamma \in \F(U,E)$ for all $\gamma\in\F(U,E)$ and the map
\[ m_h^E:\F(U,E)\to \F(U,E),\quad \gamma\mapsto h\gamma \]
is continuous.
\end{itemize}
\end{de}

\begin{re}
Since the map in (\ref{2 eq iso}) is an isomorphism of topological vector spaces, the Axioms (PB), (GL) and (MU) hold in general whenever they hold for $E=\R$. Likewise, Axiom (PF) holds in general whenever it holds for $F=\R$.
\end{re}
\noindent
Following \cite[Remark 3.5]{FGL1}, if $\mathcal{U}$ is a good collection of open subsets of $[0,\infty)^m$ and $\left( \F(U,\R)\right)_{U\in \mathcal{U}} $ is a family of locally convex space suitable for global analysis, then we have the following results.

\begin{lemma}\label{AX2}
Let $U\subseteq [0,\infty)^\infty$ be an open subset and $V,W\in \mathcal{U}$ such that $W$ has compact closure contained in $U$ and $\Theta:U\to V$ be a smooth diffeomorphism. If $\F(V,\R)$ and $\F(W,\R)$ are Fr\'echet spaces such that  $\gamma\circ \Theta|_W\in \F(W,\R)$ for all $\gamma\in \F(V,\R)$, then the map
\[ \F(\Theta|_W,\R):\F(V,\R)\to \F(W,\R),\quad \gamma\mapsto \gamma\circ \Theta|_W\]
is continuous.
\end{lemma}
\begin{proof}
Let $\gamma\in BC(V,\R)$ and $p:\R\to\R$ be a continuous seminorm, then
\[ \lVert \gamma\circ \Theta|_W \lVert_{\infty,p}:=\sup_{x\in W} p(\gamma\circ \Theta|_W(x)) \leq \sup_{z\in V} p(\gamma(z)).\]
Therefore $\gamma\circ \Theta|_W \in BC(W,\R)$. We define the continuous linear operator
\[ T:BC(\Theta|_W,\R):BC(V,\R)\to BC(W,\R),\quad \gamma\mapsto \gamma\circ \Theta|_W\]
with $\lVert T\lVert_{op}\leq 1$. Hence, its graph $\text{graph}(T)$ is closed in $BC(V,\R)\times BC(W,\R)$. Since the inclusion map $J:\F(U,\R)\to BC(U,\R)$ is continuous, we have
\[ \text{graph}(\F(\Theta|_W,\R))=(J\times J)^{-1}(\text{graph}(T)).\]
Then $\F(\Theta|_W,\R)$ is continuous by the Closed Graph Theorem.
\end{proof}

\begin{lemma}\label{AX3}
If $U\in \mathcal{U}$, $h\in C_c^{\infty}(U,\R)$ and $\F(U,\R)$ is a Fr\'echet space such that $h\gamma \in \F(U,\R)$ for all $\gamma\in\F(U,\R)$, then the map
\[ m_h:\F(U,\R)\to \F(U,\R),\quad \gamma\mapsto h\gamma \]
is continuous.
\end{lemma}
\begin{proof}
As in the previous lemma, $m_h$ is continuous since the operator 
\[M_h:BC(U,\R)\to BC(U,\R),\quad \gamma\mapsto h\gamma\] 
is continuous linear, the graph of $m_h$ is closed and therefore, $m_h$ is continuous.
\end{proof}

\begin{lemma}\label{AX4}
Let $U, V\in \mathcal{U}$ with $V\subseteq U$ and $K$ be a compact subset of $V$. Assume that, for each $\gamma\in \F(V,\R)$ with support in $K$, the map $\tilde{\gamma}:U\to \R$ defined by
\[ \tilde{\gamma}(x)=\left\{ 
\begin{array}{rl}
\gamma(x), & x\in V \\
0, & x\in U\setminus \text{supp}(\gamma) 
\end{array}
\right.\]
is in $\F(U,\R)$. If, moreover, if $\F_K(V,\R)$ is a Fr\'echet space then the map
\[ e_{U,V,K} : \F_K (V,\R)\to \F(U,\R),\quad \gamma \mapsto \tilde{\gamma} \]
is continuous
\end{lemma}
\begin{proof}
Likewise to the previous lemmas, if $BC_K(V,\R):=\{\gamma\in BC(V,\R): supp(\gamma)\subseteq K\}$ then the map
\[ BC (V,\R)\to BC_K(U,\R),\quad \gamma \mapsto \tilde{\gamma} \]
which extends functions by $0$ is a linear isometry.
\end{proof}

\begin{re}
Since a manifold with corners admits cut-off functions, we can extend the basic consequence of these axioms for the case $\Rm$ (see \cite[Section 4]{FGL1}) to our context with corners. Moreover, the proofs are exactly the same. However, the statement of Lemma \ref{F new} is new and we provide a full proof.
\end{re}

\begin{lemma}\label{pre F Acts}
Let $E$ and $F$ be finite-dimensional real vector spaces and $U, W\in \mathcal{U}$ such that $W$ is relatively compact in $U$. If $\Phi:E\to F$ is a smooth map, then $\Phi\circ \gamma|_W \in \F(W,F)$ holds for each $\gamma\in \F(U,E)$ and the map
    \[ \F(U,E)\to\F(W,F),\quad \gamma\mapsto \Phi\circ \gamma|_W\]
    is continuous. In particular, if $E=F$ and $\Phi=\text{Id}_E$, then the restriction map
    \[ \F(U,E)\to\F(W,E),\quad \gamma\mapsto \gamma|_W\]
is continuous.
\end{lemma}

\begin{lemma}\label{F new}
Let $E$ and $F$ be finite-dimensional real
vector spaces and $U,W\in \mathcal{U}$ such that $W$ is relatively
compact in~$U$.
If $V$ is an open subset of~$E$ and
\[
f\colon V\to F
\]
is a smooth map, then the map
\[
\F(U/W,f)\colon  \{\gamma\in \F(U,E)\colon \gamma(\overline{W})\subset V\}
\to \F(W,F),\quad \gamma\mapsto f\circ \gamma|_W
\]
is smooth.
\end{lemma}
\begin{proof}
Given $\gamma_0$ in the domain $D$ of $\F(U/W,f)$,
we have that $\gamma_0(\overline{W})$ is a compact subset of~$V$.
There exists a smooth function $\chi\colon V\to\R$
with compact support $K\subseteq V$ such that $\chi(y)=1$
for all $y$ in an open subset $Y\subseteq V$ with $\gamma_0(\overline{W})\subseteq Y$.
Then
\[
g\colon E\to F,\quad g(y):=\left\{
\begin{array}{cl}
\chi(y)f(y) & \mbox{\,if $\,y\in V$;}\\
0 &\mbox{\,if $y\in E\setminus K$}
\end{array}\right.
\]
is a smooth function. Since $f|_Y=g|_Y$,
we have that
\[
f\circ \gamma|_W=g\circ \gamma|_W
\]
for all $\gamma\in D$ such that $\gamma(\overline{W})\subseteq Y$,
which is an open neighborhood of $\gamma_0$ in~$D$.
To see smoothness of $\F(U/W,f)$
on some open neighborhood of $\gamma_0$
(which suffices for the proof),
we may therefore replace $f$ with $g$
and assume henceforth that $V=E$,
whence $D$ is all of $\F(U,E)$.
It suffices to show that $\F(U/W,f)$
is $C^k$ for each $k\in \N_0$,
and we show this by induction.
For the case $k=0$, see Lemma~3.1.11.
Let $k\in \N_0$ now and assume that, for all $E$, $F$, $U$, $W$ and $f\colon V\to F$
as in the lemma, with $V=E$, the map $\F(U/W,f)$ is $C^k$.
We claim that, for all $\gamma,\eta\in \F(U,E)$,
the directional derivative
\[
d\F(U/W,f)(\gamma,\eta)
\]
exists and equals $\F(U/W, df)(\gamma,\eta)$,
if we identify the locally convex spaces
$\F(U,E)\times\F(U,E)$ and $\F(U,E\times E)$;
thus
\begin{equation}\label{the-der-dir}
d\F(U/W,f)(\gamma,\eta)
=\F(U/W, df)(\gamma,\eta).
\end{equation}
If this is true, then
\[
d\F(U/W,f)=\F(U/W,df)
\]
is $C^k$ by induction and thus continuous, showing that
$\F(U/W,f)$ is $C^1$. Moreover, since $\F(U/W,f)$ is $C^1$
and $d\F(U/W,f)=\F(U/W,df)$ is $C^k$,
the map $\F(U/W,f)$ is $C^{k+1}$,
which completes the inductive proof. It only remains to prove the
claim. To this end, let $\gamma,\eta\in \F(U,E)$.
Since $\F(U/W,df)$ is continuous by the case $k=0$,
the map
\[
h\colon [0,1]\times [0,1]
\to\F(W,F),\quad (t,s)\mapsto df\circ (\gamma+st\eta,\eta)|_W=\F(U/W,df)(\gamma+st\eta,\eta)
\]
is continuous. As $\F(W,F)$ is assumed integral complete,
for each $t\in [0,1]$ the continuous path $h(t,\cdot)\colon [0,1]\to \F(W,F)$
has a weak integral
\[
I(t):=\int_0^1 df\circ (\gamma+st\eta,\eta)|_W\, ds
\]
in $\F(W,F)$. The function $I\colon [0,1]\to\F(W,F)$
is continuous by the theorem on parameter-dependent
integrals. For $0\not= t\in [0,1]$,
we consider the difference quotient
\[
\Delta(t)=\frac{\F(U/W,f)(\gamma+t\eta)-\F(U/W,f)(\gamma)}{t}=
\frac{f\circ (\gamma+t\eta)|_W+f\circ\gamma|_W}{t}.
\]
Then
\begin{equation}\label{Delta-I}
\Delta(t)=I(t).
\end{equation}
In fact, for each $x\in W$ the point evaluation
\[
\varepsilon_x\colon \F(W,F)\to F,\quad\theta\mapsto \theta(x)
\]
is a continuous linear map. It therefore commutes with the weak integral
and we obtain
\begin{eqnarray*}
I(t)(x)&=&\varepsilon_x(I(t))=\int_0^1 \varepsilon_x(df\circ (\gamma+st\eta,\eta)|_W)\, ds\\
&=&\int_0^1 df(\gamma(x)+st\eta(x),\eta(x))\, ds
=\frac{f(\gamma(x)+t\eta(x))-f(\gamma(x))}{t}\\
&=& \Delta(t)(x),
\end{eqnarray*}
applying the mean value theorem to the smooth function $f$.
Thus (\ref{Delta-I}) holds.
Exploiting the continuity of~$I$, letting $t\to 0$ we obtain
\[
\lim_{t\to 0}\Delta(t)=\lim_{t\to 0}I(t)=I(0)
=\int_0^1 df\circ (\gamma,\eta)|_W\, ds=df\circ (\gamma,\eta)|_W,
\]
establishing~(\ref{the-der-dir}). 
\end{proof}

\begin{de}
Let $U$ be a open subset of $[0,\infty)^m$ and $E$ be a finite-dimensional real vector space. We let $\F_{\text{loc}}(U,E)$ be the set of all function $\gamma:U\to E$ such that for each $V\in \mathcal{U}$ which is relatively compact in $U$ we have $\gamma|_V\in \F(V,E)$. \\
We see that each $\gamma\in \F_{\text{loc}}(U,E)$ is continuous and by the previous lemma $\F(U,E)\subseteq \F_{\text{loc}}(U,E)$. We endow $\F_{\text{loc}}(U,E)$ with the initial topology with respect to the family of restriction maps
    \[ \F_{\text{loc}}(U,E)\to\F(V,E),\quad \gamma\mapsto \gamma|_V\]
where $V\in \mathcal{U}$ which is relatively compact in $U$. This topology makes $\F_{\text{loc}}(U,E)$ a Hausdorff locally convex space.
\end{de}

\begin{lemma}\label{pre Floc Rest}
Let $E$ be a finite-dimensional vector space. If $U$ and $V$ are open subsets of $[0,\infty)^m$ such that $V\subseteq U$, then $\gamma|_V \in \F_{\text{loc}}(V,E)$ for each $\gamma\in \F_{\text{loc}}(U,E)$ and the restriction map
\[ \F_{\text{loc}}(U,E)\to \F_{\text{loc}}(V,E),\quad \gamma\mapsto \gamma|_V \]
is linear and continuous. 
\end{lemma}

\begin{lemma}\label{pre Floc Acts}
Let $E$ and $F$ be finite-dimensional real vector spaces and $U\subseteq [0,\infty)^m$ be open. If $\Phi:E\to F$ is a smooth map, then $\Phi\circ \gamma \in \F_{\text{loc}}(U,F)$ holds for each $\gamma\in \F_{\text{loc}}(U,E)$ and the map
\[ \F_{\text{loc}}(U,E)\to\F_{\text{loc}}(U,F),\quad \gamma\mapsto \Phi\circ \gamma\]
is continuous. Moreover, if $Q$ is an open subset of $E$ and $\Psi:Q\to F$ is a smooth map, then $\Psi\circ \gamma \in \F_{\text{loc}}(U,F)$ holds for each $\gamma\in \F_{\text{loc}}(U,E)$ such that $\gamma(U)\subseteq Q$.
\end{lemma}

\begin{lemma}\label{pre Floc theta}
Let $E$ be a finite-dimensional vector space, $U$ and $V$ be open subsets of $[0,\infty)^m$ and $\Theta:U\to V$ be a smooth diffeomorphism. Then $\gamma\circ \Theta \in  \F_{\text{loc}}(U,E)$ for each $\gamma\in \F_{\text{loc}}(V,E)$ and the map
\[ \F(U,E)\to \F(V,E),\quad \gamma\mapsto \gamma \circ \Theta \]
is continuous.
\end{lemma}

\begin{lemma}\label{pre glue}
Let $E$ be a finite-dimensional vector space, $U_1,...,U_n$ be open subsets of $[0,\infty)^m$ and $\gamma_j\in  \F_{\text{loc}}(U_j,E)$ for $j\in \{1,...,n\}$ such that
\[ \gamma_j|_{U_i \cap U_j} = \gamma_i|_{U_i\cap U_j},\quad \text{ for all } i,j\in\{1,...,n\}.\]
If $V\in \mathcal{U}$ is relatively compact in $U_1\cup...\cup U_n$, then $\tilde{\gamma}\in \F(V,E)$ holds for the map $\tilde{\gamma}:V\to E$ defined piecewise via $\tilde{\gamma}(x)=\gamma_j(x)$ for $x\in V\cap U_j$. \\
Moreover, if $\mathcal{E}$ is the vector subspace of $\prod_{j=1}^n \F_{\text{loc}}(U_j,E)$ given by the $n$-tuples $(\gamma_1,...,\gamma_n)$ such that $\gamma_j|_{U_i \cap U_j} = \gamma_i|_{U_i\cap U_j}$, for all  $i,j\in\{1,...,n\}$, endowed with the subspace topology, then the gluing map
\[ \text{glue}:\mathcal{E}\to \F(V,E),\quad (\gamma_1,...,\gamma_n)\mapsto \tilde{\gamma} \]
is continuous linear.
\end{lemma}

\begin{de}
Let $M$ be an $m$-dimensional compact smooth manifold with corners and $N$ an $n$ dimensional smooth manifold. Let $\F(M,N)$ be the set of all  functions $\gamma:M\to N$ such that for each $p\in M$, exist charts $\phi_M:U_M \to V_M$ of $M$ with $V_M\in\mathcal{U}$ and $\phi_N:U_N\to V_N$ a chart of $N$, such that $p\in U_M$, $\gamma(U_M)\subseteq U_N$ and $\phi_N\circ \gamma \circ \phi_M^{-1} \in \F(V_M,\R^n)$.
\end{de}

\begin{re}
For a compact smooth manifold without boundary $M$, the properties of maps between $\F$-spaces are studied in Section 5 of \cite{FGL1}. These properties can be extended to the case with corners. We recall the more important results relevant for our context.
\end{re}

\begin{lemma}
Let $M$ be an $m$-dimensional compact smooth manifold with corners, $N$ be a $n$-dimensional smooth manifold and $\gamma :M\to N$ be a continuous map. Then $\gamma\in \F(M,N)$ if and only if $\phi_N\circ \gamma\circ \phi_M^{-1}\in \F_{loc}(V_M,\R^n)$ for all charts $\phi_M:U_M \to V_M$ and $\phi_N:U_N \to V_N$ of $M$ and $N$, respectively, such that $\gamma(U_M)\subseteq U_N$.
\end{lemma}

\begin{lemma}\label{pre FM Acts}
Let $\Phi:N_1\to N_2$ be a smooth map between finite-dimensional smooth manifolds, and $M$ be a compact smooth manifold. Then $\Phi\circ \eta\in \F(M,N_2)$ for each $\eta\in \F(M,N_1)$.
\end{lemma}

\begin{re}
    Let $M$ be an $n$-dimensional compact smooth manifold with corners and $E$ be a finite-dimensional vector space. We give $\F(M,E)$ the initial topology with respect to the maps 
    \[ \F(M,E)\to \F(V_\phi,E),\quad \gamma\mapsto\gamma\circ \phi^{-1}\]
    for $\phi:U_\phi\to V_\phi$ in the maximal $C^\infty$ atlas of $M$.
\end{re}

\begin{lemma}\label{lemmatopo}
Let $M$ be a compact smooth manifold with corners and $E$ be a finite-dimensional vector space. For $i\in \{1,...,k\}$, let $\phi_i:U_i\to V_i$ be charts of $M$,  $W_i\in \mathcal{U}$ be a relatively compact subset of $V_i$ with $M=\cup_{i=1}^k \varphi_i^{-1} (W_i)$. Then the linear map
\[ \Theta:\F(M,E)\to \prod_{i=1}^k \F(W_i,E),\quad \gamma\mapsto \left(\gamma\circ \phi_i^{-1}|_{W_i}\right)_{i=1}^k \]
is a topological embedding with closed image.\\ 
The image $\text{Im}(\Theta)$ is the set $S$ of all $(\gamma_i)_{i=1}^k \in \prod_{i=1}^k \F(W_i,E)$ such that $\gamma_i \circ \phi_i (p) = \gamma_j \circ \phi_j (p)$ for all $i, j\in \{1,...,k\}$ and $p\in \phi_i^{-1} (W_i)\cap \phi_i^{-1} (W_j)$.
\end{lemma}

\begin{lemma}\label{pre FExE FExFE}
Let $M$ be an $m$-dimensional compact manifold with corners. If $E_1$ and $E_2$ are finite-dimensional vector spaces, we consider the projections $\text{pr}_j:E_1\times E_2 \to E_j$, $(x_1,x_2)\mapsto x_j$ for $j\in \{1,2\}$. Then
\[ \left( \F(M,\text{pr}_1),\F(M,\text{pr}_2)\right): \F(M,E_1\times E_2)\to \F(M,E_1)\times \F(M,E_2),\text{  }\gamma\mapsto (\text{pr}_1,\text{pr}_2)\circ \gamma\]
is an isomorphism of topological vector spaces.
\end{lemma}

\begin{lemma}\label{pre FM analytic}
If $M$ is an $m$-dimensional compact smooth manifold with corners, $E$ and $F$ are finite-dimensional $\mathbb{K}$-vector spaces for $\mathbb{K}\in \{\R,\mathbb{C}\}$, $U$ is an open subset of $E$ and $g:U\to F$ is $\mathbb{K}$-analytic, the also the map 
\[ \F(M,g):\F(M,U)\to \F(M,F),\quad \gamma\mapsto g\circ \gamma\]
is $\mathbb{K}$-analytic.
\end{lemma}

\section{Space of $\mathcal{F}$-sections}
Let $m, n\in \N$. We assume that $\mathcal{U}$ is a good collection of open subsets of $[0,\infty)^m$ and $\left(\F(U,\R)\right)_{U\in \mathcal{U}}$ is a family of locally convex spaces suitable for global analysis.\\
Let $M$ be an $m$-dimensional compact smooth manifold with corners and $N$ be an $n$-dimensional smooth manifold. For $\gamma \in \F (M,N)$ we define the set
\[ \Gamma_\F (\gamma) := \{ \sigma \in \F (M,TN) : \pi_{N} \circ \sigma = \gamma \} \]
and we endow it with the pointwise operations, making it a vector space. We make $\Gamma_\F (\gamma)$ a  Hausdorff locally convex space, using the initial topology with respect to the family of maps
\[ h_{\varphi,\phi}:\Gamma_{\F} (\gamma) \to \F(V_\varphi,\R^n) ,\quad \sigma \mapsto d\phi \circ \sigma \circ \varphi^{-1}|_W \]
where $\varphi:U_\varphi\to V_\varphi$ is a chart in the maximal $C^{\infty}$-atlas of $M$, with $W\in\mathcal{U}$ relatively compact in $V_\varphi$ and there exists a chart $\phi:U_\phi\to V_\phi$ of $N$ such that $\gamma(U_\varphi)\subseteq U_\phi$. These maps make sense because $\gamma(U_\varphi)\subseteq U_\phi$ implies $\sigma(U_\varphi)\subseteq TU_\phi$ for each $\sigma\in \gf$.

\begin{prop}\label{topembedding} Let $M$ be an $m$-dimensional compact smooth manifold with corners, $N$ be an $n$-dimensional smooth manifold and $\gamma\in \F(M,N)$. For $i\in\{1,...,k\}$, let $\varphi_i:U_i\to V_i$ be charts of $M$ such that there exists $W_i\in \mathcal{U}$ relatively compact in $V_i$ with $M=\cup_{i=1}^k \varphi_i^{-1}(W_i)$ and there exists a chart $\phi_i : U_{\phi_i}\to V_{\phi_i}$ of $N$ such that $\gamma\left(U_i\right)\subseteq U_{\phi_i}$. Then the map
\[ \Phi_\gamma:\Gamma_{\F} (\gamma) \to \prod_{i=1}^k \F(W_i,\R^n) ,\quad \sigma \mapsto \left( d\phi_i \circ \sigma \circ \varphi_i^{-1}|_{W_i} \right)_{i=1}^k \]
is a linear topological embedding with closed image given by the vector subspace of elements $(\tau_i)_{i=1}^n$ such that 
\[ \tau_{i}\circ \varphi_i(p) = d\phi_i \circ (T\phi_{j})^{-1} \Big(\phi_{j}\circ\gamma (p),\tau_j\circ \varphi_j(p)\Big)\]
for all $i,j\in\{1,...,k\}$ and $p\in \varphi_i^{-1}(W_i)\cap \varphi_j^{-1}(W_j)$.
\end{prop}
\begin{proof}
The map $\Phi_\gamma$ is continuous by definition of the topology on $\Gamma_\F(\gamma)$. We denote by $S$ the vector space of elements $(\tau_i)_{i=1}^n$ such that 
\[ \tau_{i}\circ \varphi_i(p) = d\phi_i \circ (T\phi_{j})^{-1} \Big(\phi_{j}\circ\gamma (p),\tau_j\circ \varphi_j(p)\Big)\]
for all $i,j\in\{1,..,k\}$ and $p\in \varphi_i^{-1}(W_i)\cap \varphi_j^{-1}(W_j)$. \\
Clearly $\text{Im}(\Phi_\gamma)\subseteq S$. Let $(\tau_i)_{i=1}^k \in S$, identifying the tangent bundle $TV$ with $V\times \R^n$ for any open subset $V\subseteq \R^n$, we define the function $\sigma:M\to TN$ by 
\[ \sigma(p)=T\phi_i^{-1} \left( \phi_i\circ \gamma(p) ,\tau_i \circ \varphi_i(p)\right),\quad \text{if }p\in \varphi_i^{-1} (W_i)\text{ with }i\in \{1,...,k\}.\] 
We will show that the function $\sigma$ is well defined. Let  $p\in \varphi_i^{-1}(W_i)\cap \varphi_j^{-1}(W_j)$, then
\begin{align*}
    T\phi_i^{-1} \left( \phi_i\circ \gamma(p) ,\tau_i \circ \varphi_i(p)\right) 
    &= T\phi_j^{-1}\circ T\phi_j \circ T\phi_i^{-1} \left( \phi_i\circ \gamma(p) ,\tau_i \circ \varphi_i(p)\right) \\
    &= T\phi_j^{-1}\left( \phi_j\circ \gamma(p), \tau_j\circ \varphi_j (p)\right)
\end{align*}
Hence $\sigma$ make sense. Moreover, we see that $\pi_{N}\circ \sigma = \gamma$. \\ 
For $\varphi_i|_{\varphi_i^{-1} (W_i)}: \varphi_i^{-1} (W_i) \to W_i$ and $T\phi_i : TU_{\phi_i}\to V_{\phi_i}\times \R^n$ we have
\begin{align*}
T\phi_i \circ \sigma \circ\varphi_i|_{\varphi_i^{-1} (W_i)} &=  
T\phi_i \circ \left( T\phi_i^{-1} \left( \phi_i\circ \gamma ,\tau_i \circ \varphi_i\right)\right) \circ\varphi_i|_{\varphi_i^{-1} (W_i)} \\
&=  \left( \phi_i\circ\gamma\circ \varphi_i|_{\varphi_i^{-1} (W_i)} , \tau_i \right).
\end{align*}
Since $W_i\in \mathcal{U}$ we have $\phi_i\circ \gamma\circ  \varphi_i|_{\varphi_i^{-1} (W_i)} \in \F(W_i,V_{\phi_i})$ and
\[T\phi_i \circ \sigma \circ \varphi_i|_{\varphi_i^{-1} (W_i)} \in \F(W_i,V_{\phi_i})\times \F(W_i,\R^n)\cong\F(W_i,V_{\phi_i}\times \R^n).\] 
Thus  $\sigma\in \Gamma_\F (\gamma)$. Evaluating we have
\begin{align*}
\Phi_\gamma(\sigma) &= \left( d\phi_i \circ \sigma \circ \varphi_i^{-1}|_{W_i} \right)_{i=1}^k \\
&= \left( d\phi_i \circ T\phi_i^{-1} \left( \phi_i\circ \gamma ,\tau_i\circ \varphi_i \right)\circ  \varphi_i^{-1}|_{W_i} \right)_{i=1}^k \\
&= \left( d\phi_i \circ T\phi_i^{-1} \left( \phi_i\circ \gamma\circ \varphi_i^{-1}|_{W_i} ,\tau_i \circ \varphi_i\circ  \varphi_i^{-1}|_{W_i} \right) \right)_{i=1}^k \\
&= \left( d\phi_i \left( \gamma\circ \varphi_i^{-1}|_{W_i} ,d\phi_i^{-1} \circ \tau_i\right) \right)_{i=1}^k \\
&= \left( \tau_i \right)_{i=1}^k 
\end{align*}
Hence $S\subseteq \text{Im}(\Phi_\gamma)$. The inverse of $\Phi_\gamma$ is given by $\Phi_\gamma^{-1} :\text{Im}(\Phi_\gamma)\to \Gamma_\F (\gamma)$ where
\[ \Phi_\gamma^{-1}\left((\tau_i)_{i=1}^k\right)(p)= T\phi_i^{-1} \left( \phi_i\circ \gamma(p) ,\tau_i \circ \varphi_i(p)\right),\quad \text{for }p\in \varphi_i^{-1}(W_i).\] 
Consider the arbitrary chart $\alpha:U_\alpha\to V_\alpha$ of $M$ with $W_\alpha\in\mathcal{U}$ relatively compact in $V_\alpha$ and the chart $\beta:U_\beta\to V_\beta$ of $N$ such that $\alpha(U_\alpha)\subseteq U_\beta$. We will show that
\[ h_{\alpha,\beta}\circ \Phi_\gamma^{-1} :\text{Im}(\Phi_\gamma)\to \F (W, \Rn),\quad \tau=(\tau_i)_{i=1}^k\mapsto d\beta\circ \sigma_\tau\circ \alpha^{-1}|_W \]
is continuous. Since $M=\cup_{i=1}^k \varphi_i^{-1}(W_i)$, we define the open set
\[ Q_j:=\alpha\left(U_\alpha\cap \varphi_j^{-1}(W_j)\right). \] 
Without loss of generality, we assume that there exists $r\in \{1,..,k\}$ such that $Q_j\neq \emptyset$ for each $j\in\{1,...,r\}$ and $Q_j=\emptyset$ for each $j\in\{r+1,..,k\}$.\\
Then, for $\tau \in \text{Im}(\Phi_\gamma)$ and $j\in\{1,...,r\}$ we have
\begin{align*}
d\beta \circ \sigma_\tau \circ \alpha^{-1}|_{Q_j} &= d\beta \circ T\phi_j^{-1} \left( \phi_j\circ \gamma ,\tau_j \circ \varphi_j\right) \circ \alpha^{-1}|_{Q_j} \\
&= \F( W_j,d\beta \circ T\phi_j^{-1})\circ \left(\phi_i\circ \gamma\circ \alpha^{-1}|_{Q_i}, \F(\varphi_j \circ \alpha^{-1}|_{Q_j}, \Rn) \circ  \tau_j\right).
\end{align*}
Hence $d\beta \circ \sigma_\tau \circ \alpha^{-1}|_{Q_j} \in \F_{\text{loc}} (Q_i,\Rn)$. This enables us to define the continuous map
\[ \Lambda: \text{Im}(\Phi_\gamma)\to \prod_{i=1}^r \F_{\text{loc}} (Q_i,\Rn). \quad \tau \mapsto \left(d\beta\circ \sigma_\tau \circ \alpha^{-1}|_{V_i}\right)_{i=1}^r\]
where the image set $\text{Im}(\Lambda)$ coincides with the subspace 
\[ \left\{ (\beta_1,...,\beta_r)\in \prod_{i=1}^r \F_{\text{loc}} (Q_i,\Rn): \left(\forall i,j\in\{1,...,r\}\right) \beta_i|_{Q_i\cap Q_j} = \beta_j|_{Q_i\cap Q_j} \right\}.\] 
For $(\beta_1,...,\beta_r)\in \prod_{i=1}^r \F_{\text{loc}} (Q_i,\Rn)$, we denote the gluing function
\[\beta(x):=\beta_j(x),\quad \text{if } x\in Q_i .\]
For each $W\in \mathcal{U}$ relatively compact in $Q_1\cup...\cup Q_r$ we have $\beta|_W\in\F(W,\Rn)$ and the map
\[ \text{glue}_W:\text{Im}(\Lambda)\to \F(W,\Rn),\quad (\beta_1,...,\beta_r)\mapsto \beta|_W \]
is continuous \cite[Lemma 4.1]{FGL1}. Therefore $h_{\varphi,\phi}\circ \Phi_\gamma^{-1}$ is continuous since
\[ h_{\varphi,\phi}\circ \Phi_\gamma^{-1} =  \text{glue}_W\circ \Lambda.\]
Hence $\Phi_\gamma^{-1}$ is continuous.
\end{proof}
\begin{re}
From now we consider the map $\Phi_{\gamma,P}$ as the homeomorphism 
\[\Phi_{\gamma}:\Gamma_\F (\gamma)\to \text{Im}(\Phi_{\gamma}).\]
\end{re}
\noindent

\begin{coro}
Let $\gamma\in \F(M,N)$. For $i\in\{1,...,k\}$, let $\varphi_i:U_i\to V_i$ be charts of $M$ such that there exists $W_i\in \mathcal{U}$ relatively compact in $V_i$ with $M=\cup_{i=1}^k \varphi_i^{-1}(W_i)$ and there exists a chart $\phi_i : U_{\phi_i}\to V_{\phi_i}$ of $N$ such that $\gamma\left(U_i\right)\subseteq U_{\phi_i}$. Then the space $\Gamma_\F(\gamma)$ is integral complete. Moreover:
\begin{itemize}
    \item[a)] If $\F(W_i,\Rn)$ is a Banach space for all $i\in\{1,..,k\}$, then $\Gamma_\F(\gamma)$ i a Banach space with norm $\lVert \cdot\lVert_\Gamma$ given by
    \[ \lVert \sigma \lVert_\Gamma := \sum_{i=1}^k \lVert \left(d\phi_i \circ \sigma \circ\varphi_i^{-1}|_{W_i} \right) \lVert_{\F(W_i,\Rn)},\quad \forall \sigma \in \Gamma_\F (\gamma). \]
    \item[b)]  If $\F(W_i,\Rn)$ is a Hilbert space for all $i\in\{1,..,k\}$, then $\Gamma_\F(\gamma)$ is a Hilbert space with inner product $\langle \cdot,\cdot\rangle_\Gamma$ given by
    \[ \langle \sigma,\tau \rangle_\Gamma:= \sum_{i=1}^k \left\langle  \left(d\phi_i \circ \sigma \circ \varphi_i^{-1}|_{W_i}\right), \left(d\phi_i \circ \tau \circ \varphi_i^{-1}|_{W_i}\right)\right\rangle_{\F(W_i,\Rn)},\quad \forall \sigma,\tau \in \Gamma_\F (\gamma). \]
\end{itemize}
\end{coro}

\begin{prop}\label{Gamma fon}
Let $\mathcal{U}$ be a good collection of open subsets of $[0,\infty)^m$ and $(\F(U,\R))_{U\in \mathcal{U}}$ be a family of locally convex spaces suitable for global analysis. Let $M$ be an $m$-dimensional compact smooth manifold with corners, $N_1$ and $N_2$ be finite-dimensional smooth manifolds and $\gamma\in \F(M,N_1)$. If $f:N_1\to N_2$ is a smooth function, then $Tf\circ \sigma \in \Gamma_{\F}(f\circ \gamma)$ for each $\sigma \in \Gamma_{\F}(\gamma)$. Moreover, the map
\[ \widetilde{f}:\Gamma_{\F}(\gamma)\to  \Gamma_{\F}(f\circ \gamma),\quad \sigma\mapsto Tf\circ \sigma\]
is continuous linear.
\end{prop}

\begin{proof}
Since $f$ and $Tf$ are smooth, by Lemma \ref{pre F Acts} we have $f\circ \gamma \in \F(M,N_2)$ and 
\[ Tf\circ \sigma \in \F(M,TN_2),\quad \text{ for all }\sigma\in \Gamma_\F(\gamma).\] 
Since $\sigma(t)\in T_{\gamma (t)}N_1$ for each $t\in [a,b]$, we have $T_{\gamma(t)}f \circ \sigma(t)\in T_{f\circ \gamma(t)} N_2$, thus 
\[\pi_{TN_2}\circ (Tf\circ \sigma) = f\circ \gamma\]
and $\widetilde{f}$ is well defined.\\ 
Let $\gamma\in \F(M,N_1)$. For $i\in\{1,...,k\}$, let $\varphi_{M,i}:U_{M,i}\to V_{M,i}$ be charts of $M$ such that there exists $W_{M,i}\in \mathcal{U}$ relatively compact in $V_{M,i}$ with $M=\cup_{i=1}^k \varphi_{M,i}^{-1}(W_{M,i})$ and there exist chart $\phi_{1,i} : U_{\phi_{1,i}}\to V_{\phi_{1,i}}$ and $\phi_{2,i} : U_{\phi_{2,i}}\to V_{\phi_{2,i}}$ of $N_1$ and $N_2$ respectively, such that $\gamma\left(U_{M,i}\right)\subseteq U_{\phi_{1,i}}$ and 
$f\left(U_{\phi_{1,i}}\right)\subseteq U_{\phi_{2,i}}$. We may assume that $V_{\phi_{1,i}}=\R^{n_1}$ and $V_{\phi_{2,i}}=\R^{n_2}$.\\ 
Let $\gamma_i:=\phi_{1,i}\circ \gamma|_{W_{M,i}}$ and $W_{M,i}^\prime \in \mathcal{U}$ be relatively compact in $U_{M,i}$, containing the closure of $W_{M,i}$. For $i\in\{1,...,k\}$ we define the smooth map
\[ f_i:=d\phi_{2,i} \circ Tf \circ T\phi_{1,i}^{-1}: TV_{\psi_{1,i}}\subseteq \R^{2n_1} \to \R^{n_2} \]
and
\[ G: \prod_{i=1}^k \F(W_{M,i}^\prime,\R^{2n_1}) \to \prod_{i=1}^k \F(W_{M,i},\R^{n_2}), 
\quad (\xi_i)_{i=1}^k\mapsto \left(f_i \circ (\gamma_i,\xi_i|_{W_{M,i}})\right)_{i=1}^k \]
which is continuous by Lemma \ref{pre F Acts}. We consider the linear topological embeddings
\[ \Phi_\gamma:\Gamma_{\F} (\gamma) \to \prod_{i=1}^k \F(W_{M,i},\R^{n_1}) ,\quad \sigma \mapsto \left( d\phi_{1,i} \circ \sigma \circ \varphi_{M,i}^{-1}|_{W_{M,i}} \right)_{i=1}^k \]
and 
\[ \Phi_{f\circ\gamma}:\Gamma_{\F} (f\circ\gamma) \to \prod_{i=1}^k \F(W_{M,i},\R^{n_2}) ,\quad \tau \mapsto \left( d\phi_{i,2} \circ \tau \circ \varphi_{M,i}^{-1}|_{W_{M,i}} \right)_{i=1}^k. \]
Then
\[ G(\text{Im}(\Phi_\gamma))\subseteq \text{Im}(\Phi_{f\circ \gamma}).\]
Indeed, consider $\sigma\in \gf$. Then for each $i\in\{1,...,k\}$
\[ \tau_i:=f_i(T\phi_{1,i} \circ \sigma \circ \varphi_{M,i}^{-1}|_{W_{M,i}}) = d\phi_{2,i} \circ Tf \circ \sigma \circ \varphi_{M,i}^{-1}|_{W_{M,i}}.  \]
And for all $i,j\in\{1,..,k\}$ and $p\in \varphi_{M,i}^{-1}(W_{M,i})\cap \varphi_{M,j}^{-1}(W_{M,j})$, we have
\begin{align*}
\tau_i \circ \varphi_{M,i} (p) &= d\phi_{2,i} \circ Tf\circ \sigma \circ \varphi_{M,i}^{-1}|_{W_{M,i}}\circ \varphi_{M,i}(p) \\
&= d\phi_{2,i} \circ Tf\circ \sigma (p) \\
&=  d\phi_{2,i}  \Big(f\circ \gamma (p),df \circ \sigma (p)\Big)\\
&= d\phi_{2,i} \circ (T\phi_{2,j})^{-1} \Big(\phi_{2,j}\circ f\circ\gamma (p),d\phi_{2,j} \circ df \circ \sigma\circ \varphi_{M,j}^{-1}|_{W_{M,j}}\circ \varphi_{M,j}(p) \Big)\\ 
&= d\phi_{2,i} \circ (T\phi_{2,j})^{-1} \Big(\phi_{2,j}\circ f \circ\gamma (p),\tau_j\circ \varphi_{M,j}(p)\Big) 
\end{align*}
Hence
\[ \left(f_i \circ \left(d\phi_{1,i}\circ \sigma \circ \varphi_{M,i}^{-1}|_{W_{M,i}}\right)\right)_{i=1}^k \in \text{Im}(\Phi_{f\circ \gamma}).\]
In consequence
\[ \widetilde{f} = \Phi_{f\circ\gamma}^{-1} \circ G \circ \Phi_{\gamma}. \]
Thus $\widetilde{f}$ is continuous and the linearity is clear.
\end{proof}

\begin{re}
The topology of $\Gamma_\F (\gamma)$ does not depend on the chosen family of charts. Indeed, since the identity map $\text{id}_N:N\to N$ is smooth, by previous proposition the map
\[ \widetilde{\text{id}_N}:\Gamma_\F (\gamma)\to \Gamma_\F (\text{id}_N\circ \gamma), \quad \sigma\mapsto T\text{id}_M \circ \sigma \]
is smooth regardless of chosen family of charts. Moreover, this map coincides with the identity map $\text{id}_{\Gamma}:\Gamma_\F (\gamma)\to\Gamma_\F (\gamma)$, $\sigma\mapsto\sigma$.
\end{re}

\begin{re}\label{Jgamma} 
Let $\gamma\in BC(M,N)$, we define the space of continuous sections 
\[ \Gamma_C (\gamma)=\{\sigma\in BC(M,TN): \pi_{N}\circ \sigma=\gamma\}\] 
endowed with the compact-open topology. For each $i\in\{1,...,k\}$ the inclusion map
\[ J_i:\F (W_i,\Rn)\to BC( W_i,\Rn),\quad \tau_i\mapsto \tau_i \]
is continuous, whence the inclusion map $ J:\Gamma_\F(\gamma)\to \Gamma_C (\gamma)$, $\sigma\mapsto \sigma$ is continuous.\\
This implies that the set
\[ \mathcal{V}=\{\sigma\in \Gamma_\F (\gamma) : \sigma(M)\subseteq V\} \]
is open in $\Gamma_\F (\gamma)$ for each open set $V\subseteq TN$ .
\end{re}

\begin{prop}
Let $\mathcal{U}$ be a good collection of open subsets of $[0,\infty)^m$ and $(\F(U,\R))_{U\in \mathcal{U}}$ be a family of locally convex spaces suitable for global analysis. Let $M$ be a $m$-dimensional compact smooth manifold with corners, $N_1$, $N_2$ be smooth manifolds $n_1$-dimensional and $n_2$-dimensional respectively and $\pi_j:N_1\times N_2\to N_j$ be the $j$-projection for $j\in \{1,2\}$. If $\gamma_1\in \F(M,N_1)$ and $\gamma_2\in \F(M,N_2)$  then the map
\[ \mathcal{P}:\Gamma_{\F}(\gamma_1\times \gamma_2)\to \Gamma_{\F}(\gamma_1)\times \Gamma_{\F}(\gamma_1),\quad \sigma \mapsto (T\pi_1,T\pi_2)(\sigma) \]
is a linear homeomorphism.
\end{prop}
\begin{proof}
Let $f:=(\pi_1,\pi_2)$, by Proposition \ref{Gamma fon} the map $\mathcal{P}$ is continuous and clearly linear. For $j\in\{1,2\}$ and $i\in\{1,...,k\}$, let $\varphi_{j,i}:U_{j,i}\to V_{j,i}$ be charts of $M$ such that there exists $W_{j,i}\in \mathcal{U}$ relatively compact in $V_{j,i}$ with $M=\cup_{i=1}^k \varphi_{j,i}^{-1}(W_{j,i})$ and there exists a chart $\phi_{j,i} : U_{\phi_{j,i}}\to V_{\phi_{j,i}}$ such that $\gamma_j\left(U_{j,i}\right)\subseteq U_{\phi_{j,i}}$. Then we have the homeomorphisms 
\[\Phi_{\gamma_j}:\Gamma_\F (\gamma_j)\to \text{Im}(\Phi_{\gamma_j}).\]
We consider the isomorphism of topological vector spaces
\[ \alpha :\prod_{i=1}^k\F(W_{j,i},\R^{n_1})\times \F(W_{j,i},\R^{n_2}) \to \prod_{i=1}^k \F(W_{j,i},\R^{n_1}\times \R^{n_2}), \quad (\xi_1,\xi_2)\mapsto \alpha(\xi_1,\xi_2). \]
If $\Phi_\gamma$ denotes the homeomorphism for $\gamma:=(\gamma_1,\gamma_2)\in \F(M,N_1\times N_2)$ as Proposition \ref{topembedding}, then
\[ \Theta:=\Phi_\gamma^{-1}\circ \alpha \circ \left(\Phi_{\gamma_1}\times \Phi_{\gamma_2}\right) \]
is continuous linear and for each $\sigma\in \Gamma_\F(\gamma_1)$ and $\tau\in \Gamma_\F(\gamma_2)$ we have
\[ \mathcal{P}\left(\Theta(\sigma,\tau)\right)=(\sigma,\tau).\]
Hence $\mathcal{P}$ is surjective and thus bijective, with $\mathcal{P}^{-1} = \Theta$ a continuous map.
\end{proof}

\begin{prop}\label{Fpropnog}
Let $\mathcal{U}$ be a good collection of open subsets of $[0,\infty)^m$ and $(\F(U,\R))_{U\in \mathcal{U}}$ be a family of locally convex spaces suitable for global analysis. Let $M_1$ and $M_2$ be compact smooth manifolds with corners and $N$ be a smooth manifold. If $\Theta:M_1\to M_2$ is a smooth diffeomorphism, then $\gamma \circ \Theta \in \F(M_1,N)$ for each $\gamma\in\F(M_2,N)$. Moreover, the map
\[ L_\Theta:\Gamma_\F (\gamma) \to \Gamma_{\F} (\gamma\circ \Theta), \quad \sigma \mapsto \sigma\circ \Theta\]
is linear and continuous.
\end{prop}
\begin{proof}
Let $\gamma\in\F(M_2,N)$. Let $\phi_1:U_1 \to V_1$ and $\phi_2:U_2 \to V_2$ charts of $M_1$ and $M_2$ respectively such that 
$\Theta(U_1)\subseteq U_2$. If $\phi_N:U_N \to V_N$ is a chart of $N$ such that $(\gamma\circ \Theta)(U_1)\subseteq U_N$ then
\[ \phi_N \circ \left( \gamma \circ \Theta \right) \circ \phi_1^{-1} = \phi_N \circ \gamma \circ \phi_2^{-1}\circ \phi_2 \circ \Theta\circ \phi_1^{-1}.\]
Since $\zeta:=\phi_N \circ \gamma \circ \phi_2^{-1} \in \F_{\text{loc}}(V_\psi,\Rn)$ and the map $g:=\phi_2 \circ \Theta\circ \phi_1^{-1}:V_\varphi \to V_\psi$ is a smooth diffeomorphism, by Lemma \ref{pre Floc theta} we have that $\zeta\circ g \in  \F_{\text{loc}}(V_\varphi,\Rn)$. Thus 
\[\gamma \circ \Theta \in \F(M_1,N).\]
Analogously, we can show that $\sigma\circ \Theta\in \Gamma_{\F} (\gamma\circ \Theta)$ for each $\sigma \in \gf$.\\
By compactness of $M_1$ and $M_2$, for $i\in\{1,...,k\}$ we consider charts $\phi_{1,i}:U_{1,i}\to V_{1,i}$ of $M_1$ such that there exists $W_{1,i}\in \mathcal{U}$ relatively compact in $V_{1,i}$ with $M_1=\cup_{i=1}^k \phi_{1,i}^{-1}(W_{1,i})$ and charts $\phi_{2,i}:U_{2,i}\to V_{2,i}$ of $M_2$ such that there exists $W_{2,i}\in \mathcal{U}$ relatively compact in $V_{2,i}$ with $M_2=\cup_{i=1}^k \phi_{2,i}^{-1}(W_{2,i})$ such that there exists a chart $\phi_{N,i} : U_{N.i}\to V_{N,i}$ of $N$ such that $\Theta(W_{1,i})\subseteq W_{2,i}$ and $\gamma\left(U_{2,i}\right)\subseteq U_{N,i}$. We define the topological embeddings
\[ \Phi_\gamma:\Gamma_{\F} (\gamma) \to \prod_{i=1}^k \F(W_{2,i},\R^n) ,\quad \sigma \mapsto 
\left( d\phi_{N,i} \circ \sigma \circ \phi_{2.i}^{-1}|_{W_{2,i}} \right)_{i=1}^k \]
and 
\[ \Phi_{\gamma\circ \Theta}:\Gamma_{\F} (\gamma\circ \Theta) \to \prod_{i=1}^k \F(W_{1,i},\R^n) ,\quad 
\sigma \mapsto \left( d\phi_{N,i} \circ \sigma \circ \Theta \circ \phi_{1,i}^{-1}|_{W_{1,i}} \right)_{i=1}^k \]
Since the map 
\[ \Theta_i:=\phi_{2.i}\circ \Theta\circ \phi_{1,i}^{-1}|_{W_{1,i}}:W_{1,i}\to \Theta_i(W_{2,i})\] 
is a smooth diffeomorphism, the map
\[ \F(\Theta_i,\Rn): \F_{\text{loc}}(\Theta_i(W_{2,i}),\Rn)\to \F_{\text{loc}}(W_{1,i},\Rn), \quad \tau\mapsto \tau \circ \Theta_i \]
and thus
\[ \overline{\Theta}: \prod_{i=1}^k\F_{\text{loc}}(\Theta_i(W_{2,i}),\Rn)\to \prod_{i=1}^k\F_{\text{loc}}(W_{2,i},\Rn), \quad 
(\tau_i)_{i=1}^n \mapsto \left(\tau_i \circ \Theta_i\right)_{i=1}^k \]
are continuous. We will show that $\overline{\Theta}(\text{Im}(\Phi_\gamma))\subseteq \text{Im}(\Phi_{\gamma\circ \Theta})$.\\
For each $i, j\in\{1,...,k\}$ and $\sigma\in \gf$, if
\begin{align*}
    \tau_i &:= d\phi_{N,i} \circ \sigma \circ \phi_{2,i}^{-1}|_{W_{2,i}}\circ \Theta_i \\
    &= d\phi_{N,i} \circ \sigma \circ \Theta\circ  \phi_{1,i}^{-1}|_{W_{1,i}}
\end{align*}
then
\begin{align*}
    \tau_{i}\circ \phi_{1,i}(p) 
    &= d\phi_{N,i}\circ \sigma \circ \Theta\circ \phi_{1,i}^{-1}\circ \phi_{1,i} (p) \\
    &= d\phi_{N,i}\circ \sigma \circ \Theta (p)\\
    &= d\phi_{N,i} \circ (T\phi_{N,j})^{-1} \Big(\phi_{N,j}\circ\gamma (p),d\phi_{N,j} \circ \sigma \circ \Theta(p)\Big) \\
    &= d\phi_{N,i} \circ (T\phi_{N,j})^{-1} \Big(\phi_{N,j}\circ\gamma (p),\tau_j\circ \phi_{1,j}(p)\Big)
\end{align*}
Hence $\overline{\Theta}(\text{Im}(\Phi_\gamma))\subseteq \text{Im}(\Phi_{\gamma\circ \Theta})$. In consequence, since
\[ L_\Theta = \Phi_{\gamma\circ \Theta}^{-1} \circ \left( \F(\Theta_i,\Rn) \right)_{i=1}^k \circ \Phi_\gamma\]
the map $L_\Theta$ is continuous. 
\end{proof}

\begin{prop}\label{Fgammaeval}
Let $\mathcal{U}$ be a good collection of open subsets of $[0,\infty)^m$ and $(\F(U,\R))_{U\in \mathcal{U}}$ be a family of locally convex spaces suitable for global analysis. Let $M$ be a compact manifold with corners and $N$ be a smooth manifold. Then the evaluation map
\[ \epsilon:\Gamma_\F (\gamma)\times M\to TN,\quad(\sigma,p)\mapsto \sigma(p) \]
is continuous. Moreover, for each $p\in M$ the point evaluation map
\[ \epsilon_p:\Gamma_\F (\gamma) \to TN,\quad \sigma \mapsto \sigma(p) \]
is smooth, and its co-restriction as a map to $T_{\gamma(p)}N$ is linear.
\end{prop}
\begin{proof}
Since the evaluation map
\[  \tilde{\epsilon}: \Gamma_C(\gamma)\times M\to TN, \quad (\sigma, p)\mapsto \sigma(p) \]
is continuous and the evaluation map $ \tilde{\epsilon}_p : \Gamma_C (\gamma)\to TN$, $\sigma\mapsto \sigma(p)$ is smooth for each $p\in M$ (see \cite{AGS}). Then $\epsilon = \tilde{\epsilon} \circ (J_\Gamma,\text{Id}_\R)$ and $\epsilon_p = \tilde{\epsilon}_p \circ J_\Gamma$, where $J_\Gamma:\Gamma_\F (\gamma)\to \Gamma_C(\gamma)$ is the inclusion map, which is smooth by Remark \ref{Jgamma}.
\end{proof}

\section{Manifolds of $\mathcal{F}$-maps on compact manifolds}

\begin{de}\label{localaddition}
Let $N$ be a smooth manifold and $\pi_{N}:TN\to N$ its tangent bundle. A \textit{local addition} is a smooth map ${\Sigma:\Omega\to N}$ defined on a open neightborhood $\Omega \subseteq TN$ of the zero-section ${0_N := \{ 0_p \in T_p N : p\in N\}}$ such that
\begin{itemize}
\item[a)] $\Sigma(0_p)=p$ for all $p\in N$.
\item[b)] The image $\Omega^\prime:= \big( \pi_{N},\Sigma\big) (\Omega)$ is open in $N\times N$ and the map 
\begin{equation}
    \theta_N:\Omega\to \Omega^\prime,\quad v\mapsto \big( \pi_{N}(v),\Sigma(v) \big)
\end{equation}
is a $C^\infty$-diffeomorphism.
\end{itemize}
Moreover, if $T_{0_p}(\Sigma |_{T_p N})=id_{T_p N}$ for all $p\in N$, we say that the local addition $\Sigma$ is \textit{normalized}. We denote the local addition as the pair  $(\Omega, \Sigma)$.\\
If $\theta_N:\Omega\to \Omega^\prime$ is a diffeomorphism of $\mathbb{K}$-analytic manifolds, we call ${\Sigma:\Omega\to N}$ a $\mathbb{K}$-analytic local addition.
\end{de}

\begin{re} \label{locallemma}
Let $N$ be a smooth manifold which admits a local addition. If $\pi_{TN}:T(TN)\to TN$ denotes the tangent bundle of $TN$ and $\kappa:T(TN)\to T(TN)$ its canonical flip, then $T\Sigma\circ \kappa: \tau(T\Omega)\to TN$ it is a local addition on $TN$ \cite[Lemma A.11]{AGS}. Moreover, each manifold which admits a local addition also admits a normalized local addition \cite[Lemma A.14]{AGS}. From now we will assume that each local addition is normalized.
\end{re}

\begin{re}\label{Frecharts}
Let $\mathcal{U}$ be a good collection of open subsets of $[0,\infty)^m$ and $(\F(U,\R))_{U\in \mathcal{U}}$ is a family of locally convex spaces suitable for global analysis. Let $M$ be $m$-dimensional compact manifold with corners and $N$ be a smooth manifold which admits a local addition $\Sigma:\Omega\to N$. Let $\gamma\in \F(M,N)$. We define the set
\[ \mathcal{V}_\gamma := \{ \sigma \in \Gamma_\F (\gamma) : \sigma(M)\subseteq \Omega \}. \]
which is open in $\Gamma_\F (\gamma)$ (see Remark \ref{Jgamma}) and 
\[ \mathcal{U}_\gamma := \{ \xi \in \Gamma_\F (\gamma) : (\gamma,\xi)(M) \subseteq \Omega' \}.\]
Lemma \ref{pre FM Acts} enables us to define the map
\[   \Psi_\gamma:= \F(M,\Sigma) : \mathcal{V}_\gamma  \to \mathcal{U}_\gamma,\quad \sigma  \mapsto \Sigma\circ \sigma  \]
with inverse given by
\[ \Psi_\gamma^{-1} :\mathcal{U}_\gamma \to \mathcal{V}_\gamma,\quad \xi \mapsto \theta_N^{-1}\circ (\gamma, \xi).\]
Moreover, since $M$ is compact, we note that $BC(M,N)=C(M,N)$.
\end{re}
\noindent
The following lemma is just an application of \cite[Lemma 10.1]{Ber} to our particular case.
\begin{lemma}\label{FFclosed}
Let $E$ and $F$ be finite-dimensional vector spaces, $U\subseteq E$ open and $f:U\to F$ a map. If $F_0\subseteq F$ is a vector subspace and $f(U)\subseteq F_0$, then $f:U\to F$ is smooth if and only if $f|^{F_0}:U\to F_0$ is smooth.
\end{lemma}

\begin{teo} \label{main1}
Let $\mathcal{U}$ be a good collection of open subsets of $[0,\infty)^m$ and $(\F(U,\R))_{U\in \mathcal{U}}$ be a family of locally convex spaces suitable for global analysis, then for each compact manifold $M$ with corners and smooth manifold~$N$ without boundary which admits a local addition, the set $\F(M,N)$ 
admits a smooth manifold structure
such that the sets ${\mathcal U}_\gamma$ are open in $\F(M,N)$
for all $\gamma\in \F(M,N)$
and $\Psi_\gamma\colon {\mathcal V}_\gamma\to {\mathcal U}_\gamma$
is a $C^\infty$-diffeomorphism.
\end{teo}
\begin{proof}
We endow $\F(M,N)$ with the final topology with respect to the family $\Psi_\gamma\colon {\mathcal V}_\gamma\to {\mathcal U}_\gamma$ for each $\gamma\in \F(M,N)$.
If we define the maps $\Psi_\gamma^C\colon {\mathcal V}_\gamma^C\to {\mathcal U}_\gamma^C$ on the space of continuous functions $C(M,N)$ for each $\gamma\in C(M,N)$ then the final topology on $C(M,N)$ coincides with its topology (the compact-open topology), whence the inclusion map
\[ J: \F(M,N) \to C(M,N), \quad \gamma \mapsto \gamma\]
is continuous. Moreover, since 
\[ \mathcal{U}_{J(\gamma)}^C := \{ \xi \in C(M,N) : (J(\gamma),\xi)(M) \subseteq \Omega' \}\]
is open in $C(M,N)$, the set
\[ \mathcal{U}_\gamma = \mathcal{U}_\gamma^C \cap \F(M,N) \]
is open in $\F(M,N)$.\\
The goal is to make to the family $\{ (\mathcal{U}_\gamma, \Psi_\gamma^{-1}) : \gamma \in \F (M,N) \}$ an atlas for the manifold structure.\\ 
Let $\gamma,\xi\in \F (M,N)$, it remains to show that the charts are compatible, i.e. the smoothness of the map 
\begin{equation}\label{Fcompatiblecharts}
\Lambda_{\xi,\gamma}:= \Psi_\xi^{-1} \circ \Psi_\gamma : \Psi_\gamma^{-1}(\mathcal{U}_\gamma \cap \mathcal{U}_\xi)\subseteq \Gamma_{\F} (\gamma) \to \Gamma_{\F} (\xi) 
    , \quad \sigma \mapsto \theta_N^{-1} \circ (\xi, \Sigma \circ \sigma ).
\end{equation}
For $i\in\{1,...,k\}$, let $\varphi_i:U_{M,i}\to V_{M,i}$ be charts of $M$ and $W_{M,i}\in \mathcal{U}$ such that $W_{M,i}$ is relatively compact in $V_i$ with $M=\cup_{i=1}^k \varphi_i^{-1}(W_{M,i})$ and charts $\phi_i^\gamma : U_{N,i}^\gamma\to V_{N,i}^\gamma$ and $\phi_i^\xi : U_{N,i}^\xi\to V_{N,i}^\xi$ of $N$ such that $\gamma\left( U_{M,i})\right)\subseteq U_{N,i}^\gamma$ and $\xi\left(U_{M,i}\right)\subseteq U_{N,i}^\xi$.\\ 
We will study the smoothness of the composition
\[ \Phi_\xi \circ\Lambda_{\xi,\gamma} : \Psi_\gamma^{-1}(\mathcal{U}_\gamma \cap \mathcal{U}_\xi) \to \text{Im}(\Phi_\xi )\subseteq\prod_{i=1}^k \F (W_{M,i},\R^n),\quad \sigma \mapsto \left( d\phi_i^\xi \circ \Lambda_{\xi,\gamma}(\sigma) \circ  \varphi_i^{-1}|_{W_{M,i}}  \right)_{i=1}^k\]
which is equivalent to the smoothness of $\Lambda_{\xi,\gamma}$, where $\Phi_\xi$ is the linear topological embedding as in Proposition \ref{topembedding}. By Definition \ref{pre goodcollection} c), we find $W_M^\prime$ in $\mathcal{U}$ such that $\overline{W_M^\prime}$ is relatively compact in $V_i$ and $W_{M,i}^\prime$ contains the closure of $W_{M,i}$. \\
For each $i\in\{1,...,k\}$ and $\sigma \in  \Psi_\gamma^{-1}(\mathcal{U}_\gamma \cap \mathcal{U}_\xi)$ we have
\[
    d\phi_i^\xi \circ \left(\Psi_\xi^{-1}( \Psi_\gamma (\sigma))\right) \circ  \varphi_i^{-1}|_{W_{M,i}^\prime}
    = d\phi_i^\xi \circ \theta_N^{-1} \circ \left(\xi\circ \varphi_i^{-1}|_{W_{M,i}^\prime}, \Sigma \circ \sigma\circ \varphi_i^{-1}|_{W_{M,i}^\prime}\right).
\]
Since $\sigma\left(\varphi_i^{-1}(W_{M,i}^\prime)\right) \subseteq  TU_{\phi_i^\gamma}$ we can do
\begin{align*}
\Sigma \circ \sigma\circ \varphi_i^{-1}|_{W_{M,i}^\prime} &= \Sigma\circ \left(T\phi_i^\gamma\right)^{-1} \circ T\phi_i^\gamma \circ \sigma\circ \varphi_i^{-1}|_{W_{M,i}^\prime} \\
&=\Sigma\circ \left(T\phi_i^\gamma\right)^{-1}\left(\phi_i^\gamma \circ \gamma\circ\varphi_i^{-1}|_{W_{M,i}^\prime} , d\phi_i^\gamma\circ \sigma\circ \varphi_i^{-1}|_{W_{M,i}^\prime}\right)
\end{align*}  
and 
 \[\xi\circ \varphi_i^{-1}|_{W_{M,i}^\prime} = \left(\phi_i^\xi\right)^{-1} \circ\left( \phi_i^\xi\circ \xi\circ \varphi_i^{-1}|_{W_{M,i}^\prime}\right).\]
Because all of the functions involved are continuous and have an open domain, also the composition
\begin{equation}
H_i(x,y,z) := d\phi_i^\xi \circ \theta_N^{-1} \circ \left(\left(\phi_i^\xi\right)^{-1}\left( x \right), \Sigma\circ \left(T\phi_i^\gamma\right)^{-1} \left( y , z \right) \right),
\end{equation}  
has an open domain $\mathcal{O}_i$. Hence the map $H_i: \mathcal{O}_i \to E$ is smooth. \\
By Lemma \ref{F new}, the map
\[ h_i:=\F(W_{M,i}^\prime/ W_{M,i}, H_i)\]
is smooth. By the preceding
\[\Phi_\xi\circ\Lambda_{\xi,\gamma} = h_i(\phi_i^\xi \circ \xi\circ \varphi_i^{-1}|_{W_{M,i}^\prime}, \phi_i^\gamma\circ \gamma\circ \varphi_i^{-1}|_{W_{M,i}^\prime}, d\phi_i^\gamma\circ \sigma\circ \varphi_i^{-1}|_{W_{M,i}^\prime}), \]
which is a smooth function of $\sigma$, using that the maps
 
\[ \Gamma_\F(\gamma)\to \F(W_{M,i}^\prime,\R^n),\quad \sigma\mapsto d\phi_i^\gamma\circ \sigma\circ \varphi_i^{-1}|_{W_{M,i}^\prime}\] 
are continuous linear by definition of the topology of $\Gamma_\F(\gamma)$.
Therefore $\Psi_\xi^{-1} \circ \Psi_\gamma$ is smooth. 
\end{proof}

\noindent
Proceding in the same way, using the fact that composition of $\mathbb{K}$-analytic maps is $\mathbb{K}$-analytic and using the analytic version of Lemma \ref{FFclosed} (see \cite{FGL1}), we can obtain the analogous case.

\begin{coro}
Let $\mathcal{U}$ be a good collection of open subsets of $[0,\infty)^m$ and $(\F(U,\R))_{U\in \mathcal{U}}$ be a family of locally convex spaces suitable for global analysis. For each compact manifold $M$ with corners and $\mathbb{K}$-analytic manifold~$N$ without boundary which admits a $\mathbb{K}$-analytic local addition, the set $\F(M,N)$ 
admits a $\mathbb{K}$-analytic manifold structure
such that the sets ${\mathcal U}_\gamma$ are open in $\F(M,N)$
for all $\gamma\in \F(M,N)$
and $\Psi_\gamma\colon {\mathcal V}_\gamma\to {\mathcal U}_\gamma$
is a $C^\infty$-diffeomorphism.
\end{coro}
\noindent

\begin{prop}\label{Fpropfon}
Let $\mathcal{U}$ be a good collection of open subsets of $[0,\infty)^m$ and $(\F(U,\R))_{U\in \mathcal{U}}$ be a family of locally convex spaces suitable for global analysis. Let $M$ be an $m$-dimensional compact smooth manifold with corners, $N_1$ and $N_2$ be finite dimensional smooth manifold which admits local addition. If $f:N_1\to N_2$ is a $smooth$ map, then the map
\[ \F(M,f):\F(M,N_1)\to \F(M,N_2),\quad \gamma \mapsto f\circ \gamma ,\]
is $smooth$.
\end{prop}
\begin{proof}
The map is well defined by Lemma \ref{pre FM Acts}. Let $(\Omega_1,\Sigma_1)$ and $(\Omega_2,\Sigma_2)$ be the local addition for $N_1$ and $N_2$ respectively. We consider the charts $(\mathcal{U}_\gamma, \Psi_\gamma^{-1})$ and $(\mathcal{U}_{f\circ\gamma}, \Psi_{f\circ\gamma}^{-1})$ in $\gamma \in \F(M,N)$ and $f\circ \gamma \in \F(M,N)$ respectively. We define
\[ F(\sigma):=\Psi_{f\circ\gamma}^{-1} \circ \F(M,f) \circ \Psi_\gamma (\sigma)
= (\pi_{N},\Sigma_2)^{-1} \circ \Big(f\circ \gamma , f\circ \Sigma_1 \circ \sigma \Big) \]
for all $\sigma \in \Psi_\gamma^{-1}\left( \mathcal{U}_\gamma \cap \F(M,f)^{-1}(\mathcal{U}_{f\circ\gamma})\right)$.\\ 
We will proceed as in the proof of the Theorem \ref{main1}. For $i\in\{1,...,k\}$, let $\varphi_i:U_{M,i}\to V_{M,i}$ be charts of $M$, $W_{M,i}\in \mathcal{U}$ such that $W_{M,i}$ is relatively compact in $V_i$ with $M=\cup_{i=1}^k \varphi_i^{-1}(W_{M,i})$ and $\phi_{1,i} : U_{1,i}\to V_{1,i}$ and $\phi_{2,i} : U_{2,i}\to V_{2,i}$ charts of $N_1$ and $N_2$ respectively such that $\gamma\left( U_{M,i} \right)\subseteq U_{1,i}$ and $(f\circ \gamma) (U_{M,i})\subseteq U_{2,i}$. We will study the smoothness of the composition
\[  \Phi_{f\circ\xi} \circ F : \Psi_\xi^{-1}\left( \mathcal{U}_\gamma \cap \F(M,f)^{-1}(\mathcal{U}_{f\circ\gamma})\right) \to \text{Im}(\Phi_{f\circ\xi}),\quad \sigma \mapsto \left( d\phi_{2,i} \circ F(\sigma) \circ  \varphi_{M,i}^{-1}|_{W_{M,i}}  \right)_{i=1}^k \]
where $\Phi_{f\circ\xi}$ is the linear topological embedding as in Proposition \ref{topembedding}. Using Definition \ref{pre goodcollection} c), we find sets $W_{M,i}^\prime$ in $\mathcal{U}$ which are relatively compact in $U_{M,i}$ and contain $W_{M,i}$. For each $i\in\{1,...,k\}$ and $\sigma \in \Psi_\gamma^{-1}\left( \mathcal{U}_\gamma \cap \F(M,f)^{-1}(\mathcal{U}_{f\circ\gamma})\right)$ we have
\begin{align*}
     d\phi_{2,i} \circ F(\sigma) \circ  \varphi_{M,i}^{-1}|_{W_{M,i}^\prime} &= d\phi_{2,i} \circ(\pi_{N},\Sigma_2)^{-1} \circ \Big(f\circ \gamma , f\circ \Sigma_1 \circ \sigma \Big)  \circ  \varphi_{M,i}^{-1}|_{W_{M,i}^\prime}  \\
    &= d\phi_{2,i} \circ(\pi_{N},\Sigma_2)^{-1} \circ \Big(f\circ \gamma\circ  \varphi_{M,i}^{-1}|_{W_{M,i}^\prime} , f\circ \Sigma_1 \circ \sigma \circ  \varphi_{M,i}^{-1}|_{W_{M,i}^\prime}\Big).
\end{align*}
Since $(f\circ \gamma) (U_{M,i})\subseteq U_{2,i}$ we have
\[ f\circ \gamma\circ  \varphi_{M,i}^{-1}|_{W_{M,i}^\prime} = \phi_{2,i}^{-1}\circ \left( \phi_{2,i} \circ f\circ \gamma\circ  \varphi_{M,i}^{-1}|_{W_{M,i}^\prime } \right). \]
And since $\gamma\left( \varphi_{M,i}^{-1}(V_{M,i})\right)\subseteq U_{1,i}$ we have $\sigma\left(\varphi_{M,i}^{-1}(W_{M,i})\right) \subseteq  TU_{1,i}$ whence 
\begin{align*}
f\circ\Sigma_1 \circ \sigma\circ \varphi_{M,i}^{-1}|_{W_{M,i}^\prime} &= f\circ\Sigma_1\circ T\phi_{1,i}^{-1} \circ T\phi_{1,i} \circ \sigma\circ \varphi_{M,i}^{-1}|_{W_{M,i}^\prime} \\
&=f\circ\Sigma_1\circ T\phi_{1,i}^{-1}\left(\phi_{1,i} \circ \gamma\circ\varphi_{M,i}^{-1}|_{W_{M,i}^\prime} , d\phi_{1,i}\circ \sigma\circ \varphi_{M,i}^{-1}|_{W_{M,i}^\prime}\right).
\end{align*}  
Let 
\[ H_i(x,y,z):= d\phi_{2,i} \circ (\pi_{N},\Sigma_2)^{-1} \circ (\phi_{2,i}^{-1} \left( x \right),f\circ \Sigma_1\circ T\phi_{1,i}^{-1} \left( y , z \right) ). \]
Then $H_i$ is defined on an open subset of $\R^{n_2}\times \R^{n_1}\times \R^{n_1}$ and the $\R^{n_2}$-valued function $H_i$ so obtained is smooth (because it is a composition of smooth functions).\\
By Lemma \ref{F new}, also the corresponding mappings
\[ h_i:=\F(W_{M,i}^\prime / W_{M,i}, H_i) \]
between functions spaces are smooth. By the above, we have
\[   (\Phi_{f\circ \xi}\circ F)(\sigma) h_i\left(\phi_{2,i}\circ f\circ \gamma \circ \varphi_{M,i}^{-1}|_{W_{M,i}^\prime}, \phi_{1,i}\circ \gamma\circ \varphi_{M,i}^{-1}|_{W_{M,i}^{-1}}, d\phi_{1,i}\circ \sigma \circ \varphi_{M,i}^{-1}|_{W_{M,i}^{-1}} \right) \]
which is a smooth function of $\sigma$, using that
\[ \Gamma_\F(\gamma)\to \F(W_{M,i},\R^n),\quad \sigma\mapsto d\phi_{1,i}\circ \sigma\circ \varphi_{M,i}^{-1}|_{W_{M,i}}\] 
is a continuous linear map by definition. Therefore $\F(M,f)$ is smooth.
\end{proof}
\noindent
Applying Lemma \ref{pre FM analytic} we can obtain the analogous result.

\begin{coro}\label{prop fon analytic}
Let $\mathcal{U}$ be a good collection of open subsets of $[0,\infty)^m$ and $(\F(U,\R))_{U\in \mathcal{U}}$ be a family of locally convex spaces suitable for global analysis. Let $\mathbb{K}\in \{ \R,\mathbb{C}\}$, $M$ be an $m$-dimensional compact smooth manifold with corners, $N_1$ and $N_2$ be $n$-dimensional $\mathbb{K}$-analytic manifolds with $\mathbb{K}$-analytic local additions $(\Omega_1,\Sigma_1)$ and $(\Omega_2,\Sigma_2)$ respectively. If $f:N_1\to N_2$ is a $\mathbb{K}$-analytic map, then the map
\[ \F(M,f):\F(M,N_1)\to \F(M,N_2),\quad \gamma \mapsto f\circ \gamma ,\]
is $\mathbb{K}$-analytic.
\end{coro}

\begin{re}
The manifold structures for $\F(M,N)$ given by different local additions are coincide. Indeed, since the identity map $\text{id}_N:N\to N$ is smooth, the map
\[ \F(M,\text{id}_N):\F(M,N) \to \F(M,N), \quad \gamma \to \text{id}_M \circ\gamma \]
is smooth regardless of the chosen local addition in each space. 
\end{re}

\begin{re}\label{inclusion} The inclusion map $J:\F(M,N) \to C([a,b],N)$ is smooth. Indeed, let $(\mathcal{U}_\gamma, \Psi_\gamma^{-1})$ and $(\mathcal{U}_{J(\gamma)}^C, \Psi_{J(\gamma)}^{-1})$ be charts in $\gamma\in \F(M,N)$ and $J(\gamma)\in C([a,b],N)$ respectively, then 
\[ \Psi_{J(\gamma)}^{-1} \circ J\circ \Psi_\gamma^{-1} (\sigma): \Psi_\gamma^{-1}\left(\mathcal{U}_\gamma\cap J^{-1}(\mathcal{U}_{j(\gamma)})\right)\subseteq \Gamma_\F (\gamma)\to \Gamma_C (\gamma)\] 
is a restriction of the inclusion map $\Gamma_\F (\gamma) \to \Gamma_C (\gamma)$.\\
Moreover, if $U\subseteq N$ is an open subset, then the manifold structure induced by $\F(M,N)$ on the open subset 
\[ \F(M,U):=\{ \gamma\in \F(M,N) : \gamma(M)\subseteq U \}. \]
coincides with the manifold structure on $\F(M,U)$.
\end{re}

\begin{prop}
Let $\mathcal{U}$ be a good collection of open subsets of $[0,\infty)^m$ and $(\F(U,\R))_{U\in \mathcal{U}}$ be a family of locally convex spaces suitable for global analysis. Let $M$ be a $m$-dimensional compact smooth manifold with corners, $N_1$ and $N_2$ be smooth manifolds which admit local additions, and let $\text{pr}_i:N_1\times N_2\to N_i$ be the i-th projection where $i\in \{1,2\}$, then the map
\[ \mathcal{P}:  \F(M,N_1 \times N_2)\to \F(M,N_1)\times \F(M,N_2),\quad \gamma \mapsto (\text{pr}_1,\text{pr}_2)\circ \gamma \]
is a diffeomorphism.
\end{prop}
\begin{proof}
If $(\Omega_1,\Sigma_1)$ and $(\Omega_1,\Sigma_1)$ are the local additions on $N_1$ and $N_2$ respectively, then we can assume that the local addition on $N_1\times N_2$ is 
\[ \Sigma:=\Sigma_1\times\Sigma_2: \Omega_1\times \Omega_2 \to N_1\times N_2\]
where $\Omega_1\times \Omega_2 \subseteq TN_1 \times TN_2 \cong T(N_1\times N_2)$.
The map $\mathcal{P}$ is smooth as consequence of the smoothness of the maps
\[\F(M,\text{pr}_j):\F(M,N_1 \times N_2)\to \F(M,N_i),\] 
for each ${i\in\{1,2\}}$ by the previous results. \\
Let $(\mathcal{U}_{\gamma}\times \mathcal{U}_{\gamma}, \Psi_{\gamma_1}^{-1}\times \Psi_{\gamma_2}^{-1})$ and $(\mathcal{U}_\gamma, \Psi_\gamma^{-1})$ be charts in  $(\gamma_1,\gamma_2) \in \F(M,N_1)\times \F(M,N_2)$ and $\mathcal{P}^{-1}(\gamma_1,\gamma_2)=\gamma \in \F(M,N_1\times N_2)$  respectively. Since the map
\[ \mathcal{Q}: \Gamma_{\F}(\gamma)\to \Gamma_{\F}(\gamma)\times \Gamma_{\F}(\gamma_2),\quad \tau \mapsto (\text{q}_1,\text{q}_2)\circ \tau \]
where $\text{q}_1$ and $\text{q}_2$ are the corresponding projection of the space, is an diffeomorphism of vector spaces (by Lemma 3.3 and Proposition \ref{topembedding}), we have 
\begin{align*}
  \Psi_\gamma^{-1} \circ \mathcal{P}^{-1}\circ  (\Psi_{\gamma_1}\times \Psi_{\gamma_2}) (\sigma_1,\sigma_2) 
  &= (\pi_{N_1\times N_2}, \Sigma)^{-1}\circ \left( \gamma, \mathcal{P}^{-1} \circ (\Sigma_1\times \Sigma_2)(\sigma_1,\sigma_2)\right)\\
  &= (\pi_{N_1\times N_2}, \Sigma)^{-1}\circ \left( \gamma, \Sigma\circ \mathcal{Q}^{-1}(\sigma_1,\sigma_2) \right)\\
  &= \mathcal{Q}^{-1}(\sigma_1,\sigma_2)
\end{align*}
for all  $(\sigma_1,\sigma_2) \in (\Psi_{\gamma_1}^{-1}\times \Psi_{\gamma_2}^{-1}) \left( \mathcal{U}_{\gamma_1}\times \mathcal{U}_{\gamma_2} \cap \mathcal{P}(\mathcal{U}_{\gamma})\right)$. Hence $\mathcal{P}^{-1}$ is smooth.
\end{proof}

\begin{prop}
Let $\mathcal{U}$ be a good collection of open subsets of $[0,\infty)^m$ and $(\F(U,\R))_{U\in \mathcal{U}}$ be a family of locally convex spaces suitable for global analysis. Let $M_1$ and $M_2$ be m-dimensional compact smooth manifolds with corners and $N$ be an $n$-dimensional smooth manifold which admits a local addition. If $\Theta:M_1\to M_2$ is a smooth diffeomorphism, then the map
\[ \F(\Theta,N):\F(M_2,N)\to F(M_1,N),\quad \gamma\mapsto \gamma \circ \Theta\]
is smooth.
\end{prop}
\begin{proof}
By Proposition \ref{Fpropnog} we know that the map is well defined. 
Let $(\mathcal{U}_\gamma, \Psi_\gamma^{-1})$ and $(\mathcal{U}_{\gamma\circ \Theta}, \Psi_{\gamma\circ \Theta}^{-1})$ be charts in $\gamma \in \F(M_2,N)$ and $\gamma\circ \Theta \in\F(M_1,N)$ respectively, then we have
\[ \Psi_{\gamma\circ \Theta}^{-1} \circ \F(\Theta,N) \circ \Psi_\gamma (\sigma) = 
 \theta_N^{-1}\circ (\gamma\circ \Theta, \Sigma \circ (\sigma\circ \Theta )) \]
for all $\sigma \in \Psi_\gamma^{-1}\left( \mathcal{U}_\gamma \cap \F(\Theta,N)^{-1}(\mathcal{U}_{\gamma\circ \Theta})\right)$. \\
Let $\alpha=\gamma\circ \Theta :M_1\to N$ and $\tau=\sigma\circ \Theta:M_2\to TN$, then $\tau \in \Gamma_{\F}(\alpha)$ and
\begin{align*}
\Psi_{\gamma\circ \Theta}^{-1} \circ \F(\Theta,N)\circ \Psi_\gamma (\sigma) &= 
 \theta_N^{-1}\circ (\alpha, \Sigma \circ \tau ) \\
 &= \Psi_{\alpha}^{-1}\circ \Psi_\alpha (\tau) \\
 &= \tau \\
 &= \sigma \circ \Theta.
\end{align*}
Hence, $\Psi_{\gamma\circ \Theta}^{-1} \circ\F(\Theta,N) \circ \Psi_\gamma$ is a restriction of the map 
\[ L_\Theta:{\Gamma_{\F}(\gamma)}\to {\Gamma_{\F}(\gamma\circ \Theta)},\quad \sigma\to \sigma\circ \Theta\] 
which is linear and continuous by Proposition \ref{Fpropnog}.
\end{proof}

\begin{prop}
Let $\mathcal{U}$ be a good collection of open subsets of $[0,\infty)^m$ and $(\F(U,\R))_{U\in \mathcal{U}}$ be a family of locally convex spaces suitable for global analysis. Let $M$ be an $m$-dimensional compact smooth manifold with corners. If $N$, $L$ and $K$ are smooth manifolds which admits local additions and
$f:L\times K\to N$ is a smooth map and $\gamma\in \F(M,L)$ is fixed, then
\[ f_* : \F(M,K) \to \F(M,N),\quad 	\xi\mapsto f\circ (\gamma,\xi)\]
is a smooth map.
\end{prop}
\begin{proof}
Define the smooth map
\[ C_\gamma: \F(M,K) \to \F(M,L) \times \F(M,K), \quad \xi\mapsto (\gamma,\xi).\]
Identifying  $\F(M,L) \times \F(M,K)$ with $\F(M,L\times K)$, we have 
\[f_* = \F(M,f)\circ C_\gamma.\] 
Hence $f_*$ is smooth.
\end{proof}

\begin{prop}\label{pointeval}
Let $\mathcal{U}$ be a good collection of open subsets of $[0,\infty)^m$ and $(\F(U,\R))_{U\in \mathcal{U}}$ be a family of locally convex spaces suitable for global analysis. Let $M$ be an $m$-dimensional compact smooth manifold with corners and $N$ be a $n$-dimensional smooth manifold. Then the evaluation map
\[\varepsilon:\F(M,N) \times M \to N,\quad	 (\gamma,p)\mapsto \gamma(p)\] 
is continuous. Moreover, for each $p\in M$, the point evaluation map 
\[\varepsilon_p:\F(M,N) \to N,\quad \gamma\mapsto \gamma(p)\] 
is smooth.
\end{prop}
\begin{proof}
The evaluation map
\[ \varepsilon_c:C(M,N)\times M\to N, \quad (\gamma,p)\mapsto \gamma(p))\] 
is $C^{\infty,0}$  with point evaluation $(\varepsilon_c)_p:C(M,N)\to N$, $\gamma\mapsto \gamma (p)$ smooth for each $p\in M$. Since the inclusion map $J:\F(M,N)\to C(M,N)$ is smooth, we have $\varepsilon=\varepsilon_c\circ (J,\text{Id}_M)$ and $\varepsilon_p=(\varepsilon_c)_p \circ J$ for each $p\in M$.
\end{proof}

\begin{prop}
Let $\mathcal{U}$ be a good collection of open subsets of $[0,\infty)^m$ and $(\F(U,\R))_{U\in \mathcal{U}}$ be a family of locally convex spaces suitable for global analysis. Let $M$ be a $m$-dimensional compact smooth manifold with corners and $N$ be a $n$-dimensional smooth manifold with local addition. Then, for each $q\in N$, the function $\zeta_q:M\to N$, $p\mapsto q$ is in $\F(M,N)$ and the map
\[ \zeta:N\to \F(M,N), \quad q\mapsto \zeta_q \]
is a smooth topological embedding.
\end{prop}
\begin{proof}
If $W\in \mathcal{U}$ is relatively compact
and $z\in \R^n$, consider the constant function
\[ c_z\colon W\to\R^n,\quad x\mapsto z.\]
Then $c_z\in \F(W,\R^n)$. In fact, Definition \ref{pre goodcollection} c)
provides $V\in \mathcal{U}$ such that
$\overline{W}\subseteq V$. Then
$\eta\colon V\to \R^n$, $x\mapsto 0$
is in $\F(V,\R^n)$.
The map $f\colon V\times\R^n\to\R^n$, $(x,y)\mapsto z$
is smooth, whence $c_z=f\circ (\text{id}_W,\eta|_W)\in \F(W,\R^n)$
by the pushforward axiom.\\
For each $z\in N$, the constant function
\[ \zeta_z\colon M\to N,\quad p\mapsto z\] 
is in $\F(M,N)$. In fact, if $p\in M$, $\phi_M\colon U_M\to V_M$
is chart for~$M$ around~$p$ and $\phi_N\colon U_N\to V_N$
a chart for~$N$ around $\zeta_z(p)=z$,
then Definition \ref{pre goodcollection} c) provides a relatively
compact $\phi_M(p)$-neighborhood $W\subseteq V_M$
such that $W\in \mathcal{U}$. After replacing
$\phi_M$ with its restriction to a map $\phi_M^{-1}(W)\to W$,
we may assume that $V_M\in \mathcal{U}$ and $V_M$
is relatively compact. Now
$\phi_N\circ \zeta_z\circ \phi_M^{-1}$ is the constant function
$W\to\R^n$, $x\mapsto \phi_N(z)$, which is in $\F(W,\R^n)$
as observed above. Thus $\zeta_z\in \F(M,N)$. In particular, for each $y\in N$ and $z\in T_yN$, the constant function 
\[ C_z\colon M\to T_yN,\quad v\mapsto z\]
is an element of $\F(M,T_yN)$. Since $T_yN$ is a finite-dimensional vector space,
the linear map
\begin{equation}\label{returnto}
C\colon T_yN\to \F(M,T_yN),\quad z\mapsto C_z
\end{equation}
is continuous. Given $y\in N$, consider the constant
function $\zeta_y\colon M\to N$, $p\mapsto y$, we define the vector space
\[ \Gamma_\F(\zeta_y):= \{\tau\in \F(M,TN): (\forall p\in M)~ \tau(p)\in T_{\zeta_y(p)}N=T_yN \}. \]
We show that
\[
\F(M,T_yN)\subseteq \Gamma_\F(\zeta_y)
\]
with continuous linear inclusion map.
The inclusion map $\iota\colon T_yN\to TN$
being smooth, for each $\tau\in \F(M,T_yN)$ we get
\[
\tau=\iota\circ \tau =\F(M,\iota)(\tau)\in \F(M,TN)
\]
by Lemma \ref{pre FM Acts}. Moreover, $\F(M,\iota)$ (and hence
also its co-restriction~$j$ to $\Gamma_\F(\zeta_y)$)
is continuous, by Proposition \ref{Fpropfon}. \\
Let $\Sigma\colon \Omega\to N$ be a local addition
for~$N$ and notation as in Definition \ref{localaddition}
and Remark \ref{Frecharts}.
We have $V\subseteq \Omega$ for an open
$0$-neighborhood $V\subseteq T_yN$.
Then $U_N:=\Sigma(V)$ is an open $y$-neighborhood
in~$N$ and $\psi:=\Sigma|_V^{U_N}\colon V\to U_N$
is a $C^\infty$-diffeomorphism with
\[ \psi^{-1}(u)=\theta_N^{-1}(y,u)\] 
for $u\in U_N$. If $\alpha\colon T_yN\to \R^n$ is an isomorphism
of vector spaces, then $V_N:=\alpha(V)$ is open in
$\R^n$ and
$\phi_N(u):=\alpha(\psi^{-1}(u))$ defines a chart
$\phi_N\colon U_N\to V_N$ of~$N$.
For each $v\in V_N$,
we have for each $q\in M$
\[
(\zeta_y(q),\zeta_{\phi_N^{-1}(v)}(q))=(y,\phi_N^{-1}(v))=(y,\psi(\alpha^{-1}(v)))
\in \{y\}\times
U_N\subseteq \Omega'
\]
with 
\[ \theta^{-1}_N(y,\psi(\alpha^{-1}(v)))=\psi^{-1}(\psi(\alpha^{-1}(v)))=\alpha^{-1}(v).\] 
Thus $\zeta_{\phi_N^{-1}(v)}\in \mathcal{U}_{\zeta_y}$
and
\[
\Psi_{\zeta_y}^{-1}(\zeta_{\phi_N^{-1}(v)})=\theta_N^{-1}\circ \left(\zeta_y,\zeta_{\phi_N^{-1}(v)}\right)
\]
is the constant function $C_{\alpha^{-1}(v)}$. Hence
\[
\Psi_{\zeta_y}^{-1}\circ \zeta \circ \phi_N^{-1}=j\circ C\circ \alpha^{-1}|_{V_N},
\]
which is a smooth function. Thus $\zeta$ is smooth.\\
Fix $p\in M$. The point evaluation
$\varepsilon_p\colon \F(M,N)\to N$, $\gamma\mapsto \gamma(p)$
is smooth and hence continuous.
Since $\varepsilon_p\circ \zeta=\text{id}_N$,
we deduce that $(\zeta|^{\zeta(N)})^{-1}=\varepsilon_p|_{\zeta(N)}$
is continuous.
Thus $\zeta$ is a homeomorphism onto its image.
\end{proof}

\begin{re}
Let $\mathcal{U}$ be a good collection of open subsets of $[0,\infty)^m$ and $(\F(U,\R))_{U\in \mathcal{U}}$ be a family of locally convex spaces suitable for global analysis. Let $M$ be an $m$-dimensional compact smooth manifold with corners, $N$ be an $n$-dimensional smooth manifold which admits a local addition and let $T\F(M,N)$ be the tangent bundle of $\F(M,N)$. Since the point evaluation map $\varepsilon_p:\F(M,N)\to N$ is smooth for each $p\in M$, we have
\[ T\varepsilon_p : T\F(M,N) \to TN.\] 
For each $v\in T\F(M,N)$ we define the function
\[ \Theta_N(v) : M \to TN,\quad \Theta_N(v)(p) = T\varepsilon_p (v).\]
\end{re}
\begin{prop}
Let $\mathcal{U}$ be a good collection of open subsets of $[0,\infty)^m$ and $(\F(U,\R))_{U\in \mathcal{U}}$ be a family of locally convex spaces suitable for global analysis. Let $M$ be an $m$-dimensional compact smooth manifold with corners, $N$ be an $n$-dimensional smooth manifold which admits a local addition and $\gamma\in \F(M,N)$. Then $\Theta_N(v)\in \Gamma_{\F}(\gamma)$ for each $v\in T_\gamma \F(M,N)$ and the map
\[ \Theta_\gamma : T_\gamma \F(M,N)\to \Gamma_{\F}(\gamma),\quad v\mapsto \Theta_\gamma(v):=\Theta_N|_{T_\gamma \F(M,N)}(v) \]
is an isomorphism of topological vector spaces.
\end{prop}
\begin{proof}
    Let $\Sigma:\Omega \to N$ be a normalized local addition of $N$ in sense of \cite{AGS}. Since $\Gamma_{\F}(\gamma)$ is a vector space, we identify its tangent bundle with $\Gamma_{\F}(\gamma)\times \Gamma_{\F}(\gamma)$. Let $\Psi_\gamma :\mathcal{V}_\gamma \to \mathcal{U}_\gamma$ be a chart around $\gamma$ such that $\Psi_\gamma (0) = \gamma$, then
    \[ T\Psi_\gamma : T\mathcal{V}_\gamma  \simeq \mathcal{V}_\gamma\times \Gamma_{\F}(\gamma) \to T \F(M,N) \]
    is a diffeomorphism onto its image. Moreover, 
    \[ T_0\Psi_\gamma : \{0\}\times \Gamma_{\F}(\gamma) \to T_\gamma \F(M,N) \]
    is an isomorphism of topological vector spaces. We will show that 
    \[ \Theta_\gamma \circ T\Psi_\gamma (0,\sigma) = \sigma \]
    for each $\sigma\in \Gamma_{\F}(\gamma)$. Which is equivalent to show that
    \[ T\varepsilon_p \circ T\Psi_\gamma (0,\sigma) = \sigma(p) \quad \text{ for all }p\in M.\]
    Working with the geometric point of view of tangent vectors, we see that $(0,\sigma)$ is equivalent to the curve $[s\mapsto s\sigma]$. Hence, for each $p\in M$ we have
    \begin{align*}
        T\varepsilon_p \circ T\Psi_\gamma (0,\sigma) &= T\varepsilon_p \circ T\Psi_\gamma ([s\mapsto s\sigma ]) \\
        &= T\varepsilon_p ([s\mapsto \Psi_\gamma(s\sigma)]) \\
        &= T\varepsilon_p ([s\mapsto \Sigma (s\sigma)])\\
        &=[s\mapsto \Sigma|_{T_{\gamma(p)}N} (s\sigma(t))] \\
        &= T_0\Sigma|_{T_{\gamma(p)}N} ([s\mapsto s\sigma(t)]).
    \end{align*}
Since $\Sigma$ is normalized we have $T_0\Sigma|_{T_{\gamma(p)}N}=\text{id}_{T_{\gamma(p)}N}$ and
\[ T\varepsilon_p \circ T\Psi_\gamma (0,\sigma) = \sigma (p). \]
In consequence, for each $\sigma\in \Gamma_{\F}(\gamma)$, there exists a $v\in T_\gamma \F(M,N)$ with $v=T\Psi_\gamma (0,\sigma)$ such that
\[ \Theta_\gamma (v) = \sigma.\] Moreover, the function
\[ \Theta_\gamma(v):M\to TN,\quad p\mapsto  \Theta_N(v)(p) = \sigma(p)\in T_{\gamma(p)}N \]
is in $\F(M,TN)$ with $\pi_{N}\circ \Theta_\gamma(v) =\gamma$, making the map $\Theta_\gamma$ an isomorphism of topological vector spaces. 
\end{proof}

\noindent
Following other examples of manifolds of mappings, such as the case of $C^\ell$-maps (with $\ell \geq 0$) from a compact manifold (possibly with rough boundary) to a smooth manifold which admits local addition (see e.g. \cite{AGS}), we well study the tangent bundle of $\F(M,N)$.

\begin{re}
Let $\mathcal{U}$ be a good collection of open subsets of $[0,\infty)^m$ and $(\F(U,\R))_{U\in \mathcal{U}}$ be a family of locally convex spaces suitable for global analysis. Let $M$ be an $m$-dimensional compact smooth manifold with corners and $N$ be an $n$-dimensional smooth manifold which admits a local addition. Since $TN$ admits local addition and the vector bundle $\pi_{N}:TN\to N$ is smooth, the map
\[ \F(M,\pi_{N}):\F(M,TN)\to \F(M,N),\quad \tau \mapsto \pi_{N}\circ \tau\]
is smooth. Moreover, if $\gamma \in \F(M,N)$, then
\[ \F(M,\pi_{N})^{-1}(\{\gamma\}) = \Gamma_{\F}(\gamma).\]
The following result follows the same steps as for the case of $C^\ell$-maps (with $\ell \geq 0$) from a compact manifold (possibly with rough boundary) to a smooth manifold which admits local addition \cite[Theorem A.12]{AGS}.
\end{re}

\begin{prop}
Let $\mathcal{U}$ be a good collection of open subsets of $[0,\infty)^m$ and $(\F(U,\R))_{U\in \mathcal{U}}$ be a family of locally convex spaces suitable for global analysis. Let $M$ be an $m$-dimensional compact smooth manifold with corners, $N$ be an $n$-dimensional smooth manifold which admits a local addition and $\pi_{N}:TN\to N$ its tangent bundle. Then the map
\[ \F(M,\pi_{N}):\F(M,TN)\to \F(M,N),\quad \tau \mapsto \pi_{N}\circ \tau\]
is a smooth vector bundle with fiber $\Gamma_{\F}(\gamma)$ over $\gamma\in \F(M,N)$. Moreover, the map
\[ \Theta_N:T\F(M,N)\to F(M,TN), \quad v \mapsto \Theta_N(v) \]
is an isomorphism of vector bundles.
\end{prop}

\begin{prop}
Let $\mathcal{U}$ be a good collection of open subsets of $[0,\infty)^m$ and $(\F(U,\R))_{U\in \mathcal{U}}$ be a family of locally convex spaces suitable for global analysis. Let $M$ be a $m$-dimensional compact smooth manifold with corners, $N_1$ and $N_2$ be a $n$-dimensional smooth manifold which admits a local addition. If $f:N_1\to N_2$ is a smooth map, then the tangent map of
    \[ \F(M,f):\F(M,N_1)\to \F(M,N_2),\quad \gamma\mapsto f\circ \gamma \]
    is given by 
    \[ T\F(M,f) = \Theta_{N_2}^{-1}\circ \F(M,Tf) \circ \Theta_{N_1}.\]
\end{prop}
\begin{proof}
Let $\Sigma_1:\Omega_1\to N_1$ be a local addition on $N_1$ and $\gamma \in \F(M,N_1)$.\\ 
If $\Psi_\gamma:\mathcal{V}_\gamma\to \mathcal{U}_\gamma$ is a chart on $\gamma$ such that $\Psi_\gamma (0)= \gamma$, we consider the isomorphism of vector space
    \[ T\Psi_\gamma : \{0\}\times \Gamma_{\F}(\gamma) \to T_\gamma  \F(M,N_1). \]
For $p\in M$ we denote the point evaluation in $\varepsilon_p^i:\F(M,N_i)\to N_i$ for $i\in \{1,2\}$, then for each $\sigma \in \Gamma_{\F}(\gamma)$ we have
\begin{align*}
    T\varepsilon_p^2\circ T\F(M, f) \circ T\Psi_\gamma (0,\sigma) &= T\varepsilon_p^2\circ T\F(M,f) \circ T\Psi_\gamma([s\mapsto s\sigma]) \\
    &= T\varepsilon_p^2\circ T\F(M,f) ([s\mapsto \Sigma_1 \circ s\sigma])\\
    &= T\varepsilon_p^2 ([s\mapsto f \circ \Sigma_1 \circ s\sigma])\\
    &= [s\mapsto \varepsilon_p^2\left(f \circ \Sigma_1 \circ s\sigma\right)] \\
    &= [s\mapsto f \circ \Sigma_1 (s\sigma(p))]\\
    &= Tf \circ T_0\Sigma_1|_{T_{\gamma(p)}N_1} ([s\sigma(p)]) \\
    &= Tf([s\mapsto s\sigma(p)])\\
    &= Tf(\sigma(p)) \\
    &= \F(M,Tf)\circ T\varepsilon_p^1 \circ T\Psi_\gamma (0,\sigma).
\end{align*}
Hence
\[ \Theta_{N_2}\circ T\F(M,f)  =\F(M,Tf)\circ \Theta_{N_1}.\]
\end{proof}

\begin{prop}
    Let $\mathcal{U}$ be a good collection of open subsets of $[0,\infty)^m$ and $(\F(U,\R))_{U\in \mathcal{U}}$ be a family of locally convex spaces suitable for global analysis. Let $M$ be an $m$-dimensional compact smooth manifold with corners, $N$ be an $n$-dimensional smooth manifold which admits a local addition $\Sigma:\Omega\to N$. Then the map
    \[ \F(M,\Sigma):\F(M,\Omega)\to \F(M,N),\quad \sigma\mapsto \Sigma\circ \sigma \]
    Defines a local addition on $\F(M,N)$.
\end{prop}
\begin{proof}
    For the open sets $\Omega\subseteq TN$ and $\Omega':=(\pi_N,\Sigma)(\Omega)\subseteq N\times N$ we define the open sets 
    \[ \F(M,\Omega):=\{\sigma\in \F(M,TN) : \sigma(M)\subseteq \Omega\} \]
    and 
    \[ \F(M,\Omega') := \{ \alpha\in \F(M,N\times N) : \alpha(M)\subseteq \Omega' \}. \]
    Let $\gamma\in \F(M,N)$, we define $\sigma_\gamma:M\to TN, p\mapsto 0_{\gamma(p)}$. Then $\sigma_\gamma\in \Gamma_\F (\gamma)$ and the zero-section is given by
    \[ \F(M,N)\to \F(M,TN),\quad \sigma\mapsto \sigma_\gamma.\]
    Moreover, we see that
    \[ \F(M,\Sigma)(\sigma_\gamma)(p)=\left(\Sigma\circ \sigma_\gamma\right)(p)=\Sigma(0_{\gamma(p)})=\gamma(p)\]
    hence $\F(M,\Sigma)(\sigma_\gamma)=\gamma$ for each $\gamma\in \F(M,N)$.\\
    Since $(\pi_N,\Sigma):\Omega\to \Omega'$ is a $C^\infty$-diffeomorphism, by Proposition \ref{Fpropfon}, we can define the $C^\infty$-diffeomorphism
    \[ \Theta:=\F\left(M,(\pi_N,\Sigma)\right):  \F(M,\Omega)\to \F(M,\Omega'), \quad \sigma\mapsto (\pi_{N},\Sigma)\circ \sigma \]
    with inverse given by
     \[ \Theta^{-1}:=\F\left(M,(\pi_N,\Sigma)^{-1}\right): \F(M,\Omega')\to \F(M,\Omega), \quad \alpha\mapsto (\pi_{N},\Sigma)^{-1}\circ \alpha \]
    Hence $\F(M,\Sigma)$ is a local addition on $\F(M,N)$.
\end{proof}

\begin{re}\label{Fexlie}
Let $\mathcal{U}$ be a good collection of open subsets of $[0,\infty)^m$ and $(\F(U,\R))_{U\in \mathcal{U}}$ be a family of locally convex spaces suitable for global analysis. Let $M$ be an $m$-dimensional compact smooth manifold with corners and $G$ be an $n$-dimensional Lie group, then we already know that the space $\F(M,G)$ is a Lie group (see \cite{FGL1}). We will give an alternative proof of this.\\ 
Let $e\in G$ be the neutral element, let $L_g:G\to G$, $h\mapsto gh$ be the left translation by $g\in G$ and the action 
\[ G\times TG\to TG,\quad (g,v_h)\mapsto g.v_h := TL_{g} (v_h)\in T_{gh}G.\]
If $\varphi:U\subseteq G\to V\subseteq T_e G$ is a chart in $e$ such that $\varphi(e)=0$, then the set \[\Omega_\varphi := \bigcup\limits_{g\in G}g.V \subseteq TG\] 
is open and the map
\[ \Sigma_\varphi :\Omega_\varphi \to G, \quad v\mapsto \pi_{TG}(v)\left( \varphi^{-1}( \pi_{TG}(v)^{-1}.v)\right) \]
defines a local addition for $G$, hence $\F(M,G)$ is a smooth manifold with charts constructed with $(\Omega_\varphi,\Sigma_\varphi)$. Let $\mu_G:G\times G\to G$ and $\lambda_G:G\to G$ be the multiplication map and inversion maps on $G$ respectively, we define the multiplication map $ \mu_{AC}$ and the inversion map $\lambda_{AC}$ on $\F(M,G)$ as
\[ \mu_{\F}  := \F(M,\mu_G):\F(M,G)\times \F(M,G)\to \F(M,G)\]
and
\[ \lambda_{\F} := \F(M,\lambda_G):\F(M,G) \to \F(M,G) \]
that by Lemma \ref{pre FM Acts} and Proposition \ref{Fpropfon} are smooth. \\
We observe that for the neutral element $\zeta_e:M\to G$, $p\mapsto e$ of $\F(M,G)$ we have 
\[ \Gamma_{\F}(\zeta_e) = \F(M,T_e G). \]
If $\Psi_{\zeta_e}^{-1}:\mathcal{U}_{\zeta_e}\to \mathcal{V}_{\zeta_e}$ is a chart around ${\zeta_e}\in \F(M,G)$, then we have $\mathcal{U}_{\zeta_e} = \F(M,U)$ and ${\mathcal{V}_{\zeta_{e}} = \F(M,V)}$. Moreover, we see that
\begin{align*}
\Psi_{\zeta_e} \circ\F(M,\varphi) (\gamma) &= \Sigma_\varphi \circ (\varphi\circ\gamma) \\
&= \pi_{TG}(\varphi\circ\gamma)\left( \varphi^{-1}( \pi_{TG}(\varphi\circ\gamma)^{-1}.\varphi\circ\gamma)	 \right) \\
&= e  \varphi^{-1} (e.\varphi\circ \gamma) \\
&= \gamma.
\end{align*}
This enables us to say that for the neutral element ${\zeta_e}\in \F(M,G)$ the chart is given by
\[ \F(M,\varphi):\F(M,U)\to \F(M,V),\quad \gamma \mapsto \varphi \circ \gamma. \]
\end{re}

\begin{re}
Instead of using the set $[0,\infty)^m$, it is possible to generalize all results to a good collection of open subsets $\mathcal{U}$ of a locally convex, closed subset of $\Rm$, such as half-spaces, all of $\R^m$, or a disjoint union of countably many $m$-dimensional polytopes.  
\end{re}

\section{Example of Manifolds of mappings}
Let $m\in \N$ and $\mathcal{U}$ be a good collection of open subsets of $[0,\infty)^m$. If $\left( \F(U,\R)\right)_{U\in \mathcal{U}}$ is a family of Fr\'echet spaces, then by Lemma \ref{AX4}, this family verify the Globalization Axiom. Moreover, if $\left( \F(U,\R)\right)_{U\in \mathcal{U}}$ verifies the following axioms: 
\begin{itemize}
    \item[\textbf{(PF')}] For all $U, V\in \mathcal{U}$ such that $V$ is relatively compact in $U$ and each smooth map\\ 
    $f:U\times \R\to \R$, we have $f_*(\gamma):=f\circ (\text{id}_V,\gamma|_V)\in \F(V,\R)$ for all $\gamma\in \F(U,\R)$ and the map
\[ f_*:\F(U,\R)\to \F(V,\R),\quad \gamma\mapsto f\circ (\text{id}_V,\gamma|_V)\]
is continuous.
    \item[\textbf{(PB')}] Let $U$ be an open subset of $[0,\infty)^m$ and $V,W\in \mathcal{U}$ such that $W$ has compact closure contained in $U$. Let $\Theta:U\to V$ be a smooth diffeomorphism. Then $\gamma\circ \Theta|_W\in \F(W,E)$ for all $\gamma\in \F(V,E)$.
    \item[\textbf{(MU')}] If $U\in\mathcal{U}$ and $h\in C_c^{\infty}(U,\R)$, then $h\gamma \in \F(U,E)$ for all $\gamma\in\F(U,E)$.
\end{itemize}
\noindent
Then it is a family of locally convex space suitable for global analysis.
\begin{re}
In \cite{FGL1}, Gl\"ockner and T\'arrega show that $H^{>\frac{m}{2}}(M,G)$ can be made a Lie group, where $M$ is a compact manifold of dimension $m$ (without boundary) and $G$ a finite-dimensional Lie group. This construction coincide with the construction using families of locally convex space suitable for global analysis (see Remark \ref{Fexlie}).\\
By Krikorian's work (see  \cite{Kri}) we know that the set of H\"older-continuous functions has a smooth manifold structure. In this section, we intent to construct a manifold structure for this set of mappings using the spaces of sections as modeling space.
\end{re}

\begin{de}
Let $m,n \in \N$, $0<\lambda\leq 1$ and $U \subset \Rm$ be an open and bounded subset. We say that a function $\eta:U\to \Rn$ is $\lambda$-H\"older continuous if there exists a positive constant $C$ such that
\[ \lVert \eta(x)-\eta(y) \lVert \leq C \lVert x-y\lVert^\lambda,\quad \forall x,y\in U.\]
And for each $\lambda$-H\"older continuous function we define
\[ \lVert \eta \lVert_\lambda := \sup_{ \substack{x,y\in U \\  x\neq y }} \left\{ \frac{\lVert \eta(x)-\eta(y)\lVert}{\lVert x-y\lVert^\lambda} \right\}.\]
Let $\F_\lambda(U,\Rn)$ be the space of $\lambda$-H\"older continuous functions $\eta:U\to \Rn$. By boundedness of the subset $U$, each function $\eta \in \F_\lambda(U,\Rn)$ is bounded. This allows us to consider the norm on $\F_\lambda(U,\Rn)$ 
\[ \lVert \eta \lVert_{\F_\lambda} := \lVert \eta \lVert_{\infty}+\lVert \eta \lVert_{\lambda}. \]
Then $(\F_\lambda(U,\Rn),\lVert \cdot\lVert_{\F_\lambda})$ is a Banach space (see e.g. \cite{FHOL}). In particular, if $\lambda=1$ then $\F_1(U,\R)$ denotes the space of Lipschitz continuous functions. \\
We will denote the inclusion map by $J:\F_\lambda(U,\Rn)\to BC(U,\Rn)$, which is continuous.\\
Let $\mathcal{U}$ be the family of open and bounded subsets of $\Rm$. For $0<\lambda\leq 1$ fixed, we consider the family of function spaces $\left\{ \F_\lambda(U,\R)\right\}_{U\in \mathcal{U}}$. We will show that they define a family of locally convex spaces suitable for global analysis.
\end{de}

\begin{lemma}\label{5 restric}
Let $U, V\in \mathcal{U}$ such that $V\subseteq U$. Then $\eta|_V\in \F_\lambda(V,\R)$ for each $\eta\in \F_\lambda(U,\R)$ and the map
\[ \F_\lambda(U,\R)\to \F_\lambda(V,\R),\quad \eta\mapsto \eta|_V\]
is continuous linear.
\end{lemma}
\begin{proof}
 This is direct consequence of the properties of the supremum.
\end{proof}

\begin{lemma}
    Let $U$ be an open subset of $\Rm$ and $V, W\in \mathcal{U}$ such that $W$ has compact closure contained in $U$ and $\Theta:U\to V$ be a $C^\infty$-diffeomorphism. Then $\gamma\circ \Theta|_W\in \F(W,E)$ for all $\gamma\in \F(V,E)$.
\end{lemma}
\begin{proof}
By relative compactness of $W$, we can consider a finite open cover of convex subsets $(W_i)_{i=1}^k$ for $\overline{W}$ such that $\Theta|_{W_i}$ is H\"older continuous and  $\eta\circ \Theta|_{W_i} \in  \F_\lambda(W_i,\R)$ for each $i\in\{1,...,k\}$ and $\eta\in \F_\lambda (V,\R)$. Therefore $\gamma\circ \Theta|_W\in \F(W,E)$.
\end{proof}


\begin{lemma}\label{5 h eta}
If $h\in C_c^\infty(U,\R)$, then $h\eta\in \F_\lambda(U,\R)$ for each $\eta\in \F_\lambda(U,\R)$.    
\end{lemma} 
\begin{proof}
Let $\eta \in \F_\lambda(U,\R)$. Since the function $h$ is smooth with compact support, is $\lambda$-H\"older continuous and the product $h\eta$ is in $\F_\lambda(U,\R)$.
\end{proof}

\begin{lemma}\label{5 foeta}
Let $\ell\in \N$ and $V\in \mathcal{U}$ be relatively compact. If $f:\R^\ell\to\R$ is a smooth map, then $f\circ \eta \in \F_\lambda(V,\R)$ for each $\eta \in \F_\lambda(V,\R^\ell)$ and the map
\[ \tilde{f} : \F_\lambda(V,\R^\ell)\to \F_\lambda(V,\R),\quad \eta \mapsto f\circ \eta \]
is continuous.
\end{lemma}
\begin{proof}
Let $\Delta_V$ denote the diagonal set of $V\times V$. For each $\tau\in \F_\lambda(V,\R)$, we define the function
\[ h_\tau:(V\times V)\setminus \Delta_V \to \R ,\quad (x,y)\mapsto h_\tau(x,y):=\frac{\tau(x)-\tau(y)}{\lVert x-y\lVert^{\lambda}}. \]   
Then $h_\tau \in BC((V\times V)\setminus \Delta_V,\R)$ with $\lVert h_\tau\lVert_\infty = \lVert \tau \lVert_\lambda$, hence the linear map
\[ \F_\lambda(V,\R)\to BC((V\times V)\setminus \Delta_V,\R), \quad \tau\mapsto h_\tau\]
is continuous linear. Let us consider the map
\[ H:\F_\lambda(V,\R)\to BC((V\times V)\setminus \Delta_V,\R), \quad \tau\mapsto h_\tau\]
then $H$ is continuous. This enable us to define the linear map
\[ \Phi:\F_\lambda(V,\R)\to BC(V,\R)\times BC((V\times V)\setminus \Delta_V,\R),\quad \tau\mapsto (\tau,H(\tau))\]
which is a topological embedding with closed image.
Therefore, if the map $\tilde{f}$ makes sense, its continuity is equivalent to the continuity of
\[ F: \F_\lambda(V,\R^\ell) \to BC(V,\R)\times BC((V\times V)\setminus \Delta_V,\R),\quad \eta \mapsto \left(f\circ \eta, H(f\circ\eta)\right).\]
First we will show that makes sense, i.e., $F(\eta)\in BC(V,\R)\times BC((V\times V)\setminus \Delta_V,\R)$ for each ${\eta \in \F_\lambda(V,\R^\ell)}$. Since the inclusion map $J:\F_\lambda(V,\R^\ell)\to BC(V,\R^\ell)$ and the map
\[ BC(V,\R^\ell)\to BC(V,\R),\quad \eta\mapsto f\circ \eta\]
are continuous, the first component of $F$
\[ F_1:\F_\lambda(V,\R^\ell)\to BC(V,\R),\quad \eta \mapsto f\circ \eta\]
is continuous. Let us consider the second component of $F$
\[ F_2: \F_\lambda(V,\R^\ell) \to BC((V\times V)\setminus \Delta_V,\R),\quad \eta  \mapsto H(f\circ\eta).\]
Let $\eta \in \F_\lambda(V,\R^\ell)$, then $F_2(\eta)$ is clearly continuous. We will show that $F_2(\eta)$ is bounded. For $(x,y)\in V\times V \setminus \Delta_V$ we have
\[ F_2(\eta)(x,y) = H(f\circ \eta)(x,y) = \frac{f(\eta(x))-f(\eta(y))}{\lVert x-y\lVert^\lambda}. \]
Since $V$ is relatively compact, the set $\eta(V)$ can be contained on an open ball $B_{R_\eta}(0)$ for a constant $R_\eta>0$ large enough. By smoothness, the map $f$ verifies
\[ \lvert f(u)-f(v)\lvert\leq L_{f,\eta} \lVert u-v \lVert,\quad u,v\in \overline{B_{R_\eta}(0)}, \]
for some constant $L_{f,\eta}>0$. Therefore
\[ \lVert F_2(\eta)\lVert_\infty \leq L_{f,\eta} \lVert \eta \lVert_\lambda. \]
Then $F_2(\eta)\in BC((V\times V)\setminus \Delta_V,\R)$. Now we will show that $F_2$ is continuous in $\eta\in F_\lambda(V,\R^\ell)$. Let $\delta>0$ and $\gamma\in F_\lambda(V,\R^\ell)$ such that 
\[\lVert \eta - \gamma\lVert_{\F_\lambda} := \lVert \eta-\gamma\lVert_\infty + \lVert \eta-\gamma\lVert_\lambda\leq \delta.\] 
Then for each $z\in V$ we have 
\[ \lVert \eta(z)-\gamma(z) \lVert \leq \delta, \]
which mean that $\gamma(z)\in B_\delta(\eta(z))$. Therefore
\[ \gamma(V)\subseteq \bigcup_{z\in V} B_\delta (\eta(z)).\]
Let $R_\eta>0$ the constant which verifies $\eta(V)\subseteq B_{R_\eta}(0)$, then $B_\delta(\eta(z))\subseteq B_{R_\eta+\delta}(0)$ for each $z\in V$. In consequence, $\gamma(V)$ and $\eta(V)$ are contained in $B_{R_\eta+\delta}(0)$ and by smoothness of $f$, there exists a constant $G_{f,\eta}>0$ such that
\[ \lvert df(u_1,v_1)-df(u_2,v_2)\lvert\leq G_{f,\eta} \lVert (u_1,v_1)- (u_2,v_2) \lVert=G_{f,\eta}(\lVert u_1-u_2\lVert+\lVert v_1-v_2\lVert),\]
for each $(u_1,v_1), (u_2,v_2)\in \overline{B_{R_\eta+\delta}(0)}\times \overline{B_{R_\eta+\delta}(0)}$. By the mean value theorem, we have
\[ f(u_1)-f(u_2)=\int_0^1 df(u_2+t(u_1-u_2),u_1-u_2)dt,\quad u_1, u_2\in \overline{B_{R_\eta+\delta}}.\]
Hence, if $\omega:= \lvert F_2(\eta)(x,y)-F_2(\gamma)(x,y) \lvert$ then
\begin{align*}
\omega&= \left\lvert \frac{f(\eta(x))-f(\eta(y))}{\lVert x-y\lVert^\lambda}-\frac{f(\gamma(x))-f(\gamma(y))}{\lVert x-y\lVert^\lambda}\right\lvert \\
&= \left\lvert\int_0^1\left( df\left(\eta(y)+t(\eta(x)-\eta(y)),\frac{\eta(x)-\eta(y)}{\lVert x -y\lVert^\lambda}\right)- df\left(\gamma(y)+t(\gamma(x)-\gamma(y)),\frac{\gamma(x)-\gamma(y)}{\lVert x -y\lVert^\lambda}\right)\right)dt\right\lvert\\
&\leq  G_{f,\eta} \int_0^1  \left\lVert \left(\eta(y)+t(\eta(x)-\eta(y)),\frac{\eta(x)-\eta(y)}{\lVert x -y\lVert^\lambda}\right)-\left(\gamma(y)+t(\gamma(x)-\gamma(y)),\frac{\gamma(x)-\gamma(y)}{\lVert x -y\lVert^\lambda}\right)\right\lVert dt\\
&\leq G_{f,\eta}\left( \int_0^1 \lVert t(\eta(x)-\gamma(x))+(1-t)(\eta(y)-\gamma(y))\lVert dt + \frac{\lVert\big(\eta(x)-\gamma(x)\big)-\big(\eta(y)-\gamma(y)\big)\lVert}{\lVert x-y\lVert^\lambda} \right) \\
&\leq G_{f,\eta}( \lVert \eta-\gamma \lVert_\infty + \lVert \eta-\gamma \lVert_\lambda ) \\
& \leq G_{f,\eta}\delta.
\end{align*} 
If $\varepsilon=G_{f,\eta}\delta$, we have 
\[ \lVert F_2(\eta)-F_2(\gamma)\lVert_\infty \leq \varepsilon.\] Therefore, the map $F_2$ is continuous and in consequence, the map $\tilde{f}$ is continuous.
\end{proof}

\begin{lemma}\label{5 pf}
Let $U, V\in \mathcal{U}$ such that $V$ is relatively compact in $U$. If $f:U\times \R^n\to \R$ is a smooth function, then $f_*(\eta):=f\circ (\text{id},\eta|_V)\in \F_\lambda (V,\R)$ for all $\eta \in \F_\lambda(U,\R^n)$ and the map
\[ f_*:\F_\lambda (U,\R^n)\to \F_\lambda (V,\R),\quad \eta\mapsto f_*(\eta)=f\circ (\text{id},\eta|_V)\]
is continuous.
\end{lemma}
\begin{proof}
First let assume that $U=\R^m$. Let $\text{id}:V\to \R^m$ be the identity map, then $\text{id}\in \F_\lambda(V,\R^m)$ and by Lemma \ref{5 restric}, the map
\[ F_\lambda (\R^m,\R^n)\to F_\lambda (V,\Rm\times \R^n),\quad \eta\mapsto (\text{id},\eta|_V) \]
is continuous. If $\ell=m+n$, by Lemma \ref{5 foeta}, the map
\[ F_\lambda (V, \Rm\times \R^n)\to F_\lambda (V,\R),\quad \beta \mapsto f\circ  \beta\]
is continuous. Therefore $f_*$ is just the composition of continuous mappings. \\
Let assume that $U\neq \Rm$. Let $\chi:\Rm\to\R$ be a cut-off function for $\overline{V}$ supported in $U$ (see e.g. \cite[Proposition 2.25]{JLee}); we define
\[ g:\R^m\times \R^n\to \R,\quad (x,y) \mapsto 
\left\{\begin{array}{ll}
 \chi(x)f(x,y),& \text{if } x\in U \\
 0, & \text{if } x\in \R^m\setminus\text{supp}(\chi)
\end{array}\right.  \]
Then $g$ is smooth and, as before, the map
\[ g_*:\F_\lambda (\R^m\times \R^n,\R^n)\to \F_\lambda (V,\R),\quad \eta\mapsto g_*(\eta)=g\circ (\text{id},\eta|_V)\]
is continuous. Moreover, for each $\eta \in \F_\lambda(U,\R^n)$ and $x\in V$ we have
\begin{align*}
    g_*(\eta)(x)&=g\circ(id,\eta|_V)(x) \\
    &= g(x,\eta|_V(x))  \\
    &= \chi(x)f(x,\eta|_V (x)) \\
    &= f(x,\eta|_V(x)) \\
    &= f_*(\eta)(x),
\end{align*}
whence $g_*=f_*$.
\end{proof}

\begin{re}
By Lemma \ref{5 pf}, the axiom (PF') is verified.
\end{re}
\noindent
Combining all these lemmas, we can conclude with the following Lemma.
\begin{lemma}\label{5 family}
Let $m\in \N$, $\mathcal{U}$ be the collection of open subsets of $\Rm$ and $0<\lambda\leq 1$. Then the family of Banach spaces $\{\F_\lambda(U,\R)\}_{U\in \mathcal{U}}$ define a family of locally convex spaces suitable for global analysis,
\end{lemma}

\begin{de}
Let $M$ and $N$ be finite-dimensional smooth manifolds without boundary and $0<\lambda\leq 1$. We denote the set $C^{0,\lambda}(M,N)$ of all functions $\gamma:M\to N$ such that for each $p\in M$, there exist the charts $\phi_M:U_M\to V_M$ of $M$ and $\phi_N:U_N\to V_N$ of $N$, such that $p\in U_M$, $\gamma(U_M)\subseteq U_N$ and $\phi_N\circ \gamma\circ \phi_M^{-1}\in \F_\lambda(V_M,\Rn)$.
\end{de}
\noindent
By Lemma \ref{5 family} we conclude.

\begin{prop}
Let $0<\lambda\leq 1$. For each compact manifold $M$ without boundary and smooth manifold N without boundary which admits local addition, the set $C^{0,\lambda}(M,N)$ admits a smooth manifold structure. 
\end{prop}

\begin{re}
Let $N_1$ and $N_2$ be finite-dimensional smooth manifolds without boundary which admit local additions. If $f:N_1\to N_2$ is a smooth map, then by Proposition \ref{Fpropfon}, the map
\[ C^{0,\lambda}(M,N_1)\to C^{0,\lambda}(M,N_2),\quad \gamma\mapsto f\circ \gamma \]
is smooth. 
\end{re}

\begin{prop}
Let $M$ be a compact smooth manifold without boundary and $N$ a smooth manifold without boundary which admits a local addition. If $0<\beta\leq \lambda\leq 1$, then ${\gamma \in C^{0,\beta}(M,N)}$ for each $\gamma\in C^{0,\lambda}(M,N)$. Moreover, the map
\[ \iota:C^{0,\lambda}(M,N)\to C^{0,\beta}(M,N),\quad \gamma\mapsto \gamma\]
is smooth.
\end{prop}
\begin{proof}
Let $\gamma\in C^{0,\lambda}(M,N)$, then for each $p\in M$, there exists the charts $\phi_M:U_M\to V_M$ of $M$ and $\phi_N:U_N\to V_N$ of $N$, such that $p\in U_M$, $\gamma(U_M)\subseteq U_N$ and $\phi_N\circ \gamma\circ \phi_M^{-1}\in \F_\lambda(V_M,\Rn)$. For each $U\in \mathcal{U}$, it is known that for $\beta\leq \lambda$  the linear operator
\[ I_{U}: \F_\lambda(U,\Rn)\to \F_\beta(U,\Rn),\quad \tau \mapsto \tau \] 
is continuous. In particular, we have
\[I_{V_M}(\phi_N\circ \gamma\circ \phi_M^{-1})=\phi_N\circ \gamma\circ \phi_M^{-1}\in \F_\beta(V_M,\Rn).\] 
Therefore $\gamma\in C^{0,\beta}(M,N)$.  Now, we consider the charts $(\mathcal{U}_\gamma, \Psi_\gamma^{-1})$ and $(\mathcal{U}_{\iota(\gamma)}, \Psi_{\iota(\gamma)}^{-1})$ in $\gamma \in C^{0,\lambda}(M,N)$ and $\iota(\gamma) \in C^{0,\beta}(M,N)$ respectively, then the map
\[ \Psi_{\iota(\gamma)}^{-1} \circ\iota \circ \Psi_\gamma:  \Psi_\gamma^{-1}\left( \mathcal{U}_\gamma \cap \iota^{-1}(\mathcal{U}_{\iota(\gamma)})\right)\to \Psi_{\iota(\gamma)}\left( \mathcal{U}_\gamma \cap \iota^{-1}(\mathcal{U}_{\iota(\gamma)})\right)\]
given by
\[ \Psi_{\iota(\gamma)}^{-1} \circ\iota \circ \Psi_\gamma (\sigma)
= (\pi_{N},\Sigma_N)^{-1} \circ \Big(\iota(\gamma) , \iota (\Sigma_N \circ \sigma) \Big) \]
is just a restriction of the map
\[ \tilde{\iota}:\Gamma_{\F_\lambda}(\eta)\to \Gamma_{\F_\beta}(\iota(\eta)),\quad \sigma\mapsto \sigma,\]
which is continuous by Proposition \ref{topembedding} and continuity of the maps $\{I_U\}_{U\in\mathcal{U}}$. 

\end{proof}

\noindent
{\bf Acknowledgements}: The author would like to thank Helge Gl\"ockner for his guidance during
the development of this work. The author was partially supported by the FONDECYT Grant \#1241719 and by ANID and DAAD (DAAD/BecasChile 2020, ID:91762237/62190017).


\bigskip

\begin{thebibliography}{99}
\bibitem{ASm}
Alzaareer, H. and A. Schmeding,
\emph{Differentiable mappings on products with different
degrees of differentiability in the two factors},
Expo.\ Math.\ \textbf{33},
(2015); 184--222.

\bibitem{Alz}
Alzaareer, H., \emph{Differential calculus on multiple products}, 
Indag. Math. 30 (2019), 1036-1060

\bibitem{AGS}
Amiri, H., H. Gl\"{o}ckner, and A. Schmeding,
\emph{Lie groupoids of mappings taking values in a Lie groupoid},
Arch.\ Math.\ (Brno) \textbf{56} (2020),
307-356. 

\bibitem{Ber}
Bertram W., H. Gl\"{o}ckner and K.-H. Neeb,
\emph{Differential calculus over general base fields and rings}, Expo. Math. \textbf{22} (2004), 213-282.

\bibitem{Bni}
Bastiani, A., \emph{Applications diff\'erentiables et vari\'et\'es
diff\'erentiables 
de dimension infinie}, J. Anal.\ Math.\ \textbf{13} (1964),
1--114.

\bibitem{Bochnak1971}
Bochnak J. and J. Siciak, Analytic functions in topological vector spaces, \emph{Stud. Math.}, Stud. Math. 39 (1971), 77-112.

\bibitem{Els}
Eells, J. Jr.,
\emph{A setting for global analysis},
Bull.\ Amer.\ Math.\ Soc.\ \textbf{72} (1966),
751--807.

\bibitem{FHOL}
Fiorenza R.,
``H\"older and locally H\"older Continuous Functions, and Open Sets of Class $C^k, C^{k,\lambda}$'',
Birkh\"auser, 2016.

\bibitem{FKl}
Flaschel, P. and W. Klingenberg,
``Riemannsche
Hilbertmannigfaltigkeiten,''
Springer,
Berlin, 1972.

\bibitem{FPM}
Florencio, M., F. Mayoral, and P. J. Pa\'{u}l,
\emph{Spaces of vector-valued integrable functions and
localization of bounded subsets},
Math.\ Nachr.\ \textbf{174} (1995), 89--111.

\bibitem{Floret1971}
 Floret, K., \emph{Lokalkonvexe Sequenzen mit kompakten Abbildungen}, J. Reine Angew. Math. 247 (1971), 155-195.

\bibitem{GL1}
Gl\"{o}ckner, H.,
\emph{Infinite-dimensional
Lie groups without completeness restrictions},
pp.\ 43--59 in:
Strasburger, A. et al.\ (eds.),
``Geometry and Analysis on Finite- and Infinite-Dimensional
Lie Groups,''
Banach Center Publ.\ \textbf{55},
Warsaw, 2002.

\bibitem{GL3}
Gl\"{o}ckner, H.,
\emph{Lie groups of real analytic diffeomorphisms are $L^1$-regular},
Nonlinear Anal. Volume 252 (2025), 113690.

\bibitem{GL5}
Gl\"{o}ckner, H.,
\emph{Smoothing operators for vector-valued functions
and extension operators},
preprint,
arXiv:2006.00254

\bibitem{FGL1}
 Gl\"ockner H. and L. T\'arrega, \emph{Mapping groups associated with real-valued function spaces and direct limits of Sobolev-Lie groups}, J. Lie Theory 33 (2023), 271-296.

\bibitem{Kelley1975}
 Kelley J. L., \emph{General Topology}, Springer, New York, 1975.

\bibitem{Krg}
Klingenberg, W. P. A.,
``Riemannian Geometry,''
de Gruyter,
Berlin,
1995. 

\bibitem{KMr}
Kriegl, A. and P. W. Michor,
``The Convenient Setting of Global
Analysis,''
Amer.\ Math.\ Soc.,
Providence, 1997.

\bibitem{Kri}
Krikorian, N.,
\emph{Differentiable structures on function spaces},
Trans.\ Amer.\ Math.\ Soc.\ \textbf{171}
(1972), 67--82. 

\bibitem{Lang}
Lang, S., \emph{``Fundamentals of Differential Geometry''}, Springer, 1999.

\bibitem{JLee}
Lee, J., \emph{``Introduction to Smooth Manifolds''}, Springer, 2000.

\bibitem{Mor}
Michor, P. W.,
\emph{``Manifolds of Differentiable
Mappings''},
Shiva, Orpington,
1980.

\bibitem{MCM}
Michor, P. W.,
\emph{Manifolds of mappings for continuum mechanics},
pp.\ 3--75 in:
R. Segev et al.\ (eds.), ``Geometric Continuum Mechanics,''
Birkh\"{a}user, Cham, 2020.

\bibitem{Pls}
Palais, R. S.,
``Foundations of Global Non-Linear
Analysis,''
W. A. Benjamin, New York,
1968.

\bibitem{Pin}
Pinaud, M. ``Manifold of mappings and regularity properties of half-Lie groups,'' doctoral dissertation, Paderborn University, 2025 (see nbn-resolving.org/urn:nbn:de:hbz:466:2-54221)

\bibitem{Rudin1991}
W. Rudin, \emph{Functional Analysis}, McGraw Hill, 1991.
\end{thebibliography}
\end{document}